\documentclass[reqno]{amsart}
\usepackage{amsmath,amsthm,amsfonts,amssymb,amscd}
\usepackage{enumerate}
\usepackage[vcentermath,enableskew]{youngtab}
\usepackage{microtype}
\usepackage{tikz}
\usepackage{mathrsfs}
\usepackage[ocgcolorlinks, linkcolor=red!80!black, citecolor=blue!80!black]{hyperref}
\usepackage{mathtools}

\newtheorem{thm}{Theorem}[section]
\newtheorem*{thm*}{Theorem}
\newtheorem{lem}[thm]{Lemma}
\theoremstyle{definition} \newtheorem{defn}[thm]{Definition}
\newtheorem*{defn*}{Definition}
\theoremstyle{definition} \newtheorem{ex}[thm]{Example}
\newtheorem*{lem*}{Lemma}
\newtheorem{cor}[thm]{Corollary}
\newtheorem*{cor*}{Corollary}
\theoremstyle{definition} \newtheorem{rem}[thm]{Remark}
\theoremstyle{definition} \newtheorem{alg}[thm]{Algorithm}
\newtheorem*{warn}{Warning}
\newtheorem*{conj*}{Conjecture}

\newtheorem{prop}[thm]{Proposition}

\newcommand{\CC}{\mathbb{C}}
\newcommand{\NN}{\mathbb{N}}
\newcommand{\ZZ}{\mathbb{Z}}
\newcommand{\RR}{\mathbb{R}}

\newcommand{\A}{\mathcal{A}}
\let\sectionsymbol\S 
\renewcommand{\S}{\mathfrak{S}}
\newcommand{\I}{\mathcal{I}}
\renewcommand{\th}{\textsuperscript{th}\,}

\DeclareMathOperator{\GL}{GL}
\DeclareMathOperator{\Fl}{Fl}

\newcommand{\Clan}{\mathrm{Clan}}
\newcommand{\clRed}{\mathcal{R}}
\newcommand{\Red}{\mathcal{R}}
\newcommand{\clS}{\mathfrak{S}}
\newcommand{\WYS}{\mathfrak{S}^{\textsc{wy}}}
\newcommand{\ush}{\operatorname{ush}}
\newcommand{\lsh}{\operatorname{lsh}}
\DeclareMathOperator{\comp}{comp}
\newcommand{\nred}{\mathrm{R}}
\DeclareMathOperator{\rev}{rev}
\renewcommand{\neg}{\operatorname{neg}}

\renewcommand{\det}{\operatorname{det}}
\renewcommand{\epsilon}{\varepsilon}

\newcommand{\st}{\operatorname{st}}
\newcommand{\SYT}{\operatorname{SYT}}

\begin{document}

\date{\today}
\author{Brian Burks}
\thanks{The first author was supported by NSF grant DMS-1600223.}
\author{Brendan Pawlowski}
\thanks{}
\title{Reduced words for clans}

\begin{abstract}
Clans are combinatorial objects indexing the orbits of $\GL(\mathbb{C}^p) \times \GL(\mathbb{C}^q)$ on the variety of flags in $\mathbb{C}^{p+q}$. This geometry leads to a partial order on the set of clans analogous to weak Bruhat order on the symmetric group, and we study the saturated chains in this order. We prove an analogue of the Matsumoto-Tits theorem on reduced words in a Coxeter group. We also obtain enumerations of reduced word sets for particular clans in terms of standard tableaux and shifted standard tableaux.
\end{abstract}

\maketitle

\section{Introduction}

For $p, q \in \NN$, a $(p,q)$-\emph{clan} is an involution in the symmetric group $S_{p+q}$, each of whose fixed points is labeled either $+$ or $-$,  for which
    \begin{equation*}
        \text{(\# of fixed points labeled $+$)} - \text{(\# of fixed points labeled $-$)} = p-q.
    \end{equation*}

We draw clans as partial matchings of $[n] := \{1,2,\ldots,n\}$ where $n = p+q$:
\begin{ex} \label{ex:clan-pictures}
    The $(1,2)$-clans are
    \begin{center}
    \begin{tabular}{cccccc}
        \begin{tikzpicture}
            \node (353741629) at (0.0,0) {$+$}; 
            \node (353741630) at (0.333,0) {$-$}; 
            \node (353741631) at (0.667,0) {$-$}; 
        \end{tikzpicture}  &
        \begin{tikzpicture} 
            \node (278017233) at (0.0,0) {$-$}; 
            \node (278017234) at (0.333,0) {$+$}; 
            \node (278017235) at (0.667,0) {$-$}; 
        \end{tikzpicture}    &
        \begin{tikzpicture} 
            \node (768271227) at (0.0,0) {$-$}; 
            \node (768271228) at (0.333,0) {$-$}; 
            \node (768271229) at (0.667,0) {$+$}; 
        \end{tikzpicture}    &
        \begin{tikzpicture}  
            \node[circle, fill, inner sep=0pt, minimum size = 1mm] (99440994) at (0.0,0) {}; 
            \node[circle, fill, inner sep=0pt, minimum size = 1mm] (99440995) at (0.333,0) {}; 
            \node (99440996) at (0.667,0) {$-$}; 
            \draw (99440994) to[bend left=40] (99440995); 
        \end{tikzpicture}    &
        \begin{tikzpicture} 
            \node (20148541) at (0.0,0) {$-$}; 
            \node[circle, fill, inner sep=0pt, minimum size = 1mm] (20148542) at (0.333,0) {}; 
            \node[circle, fill, inner sep=0pt, minimum size = 1mm] (20148543) at (0.667,0) {}; 
            \draw (20148542) to[bend left=40] (20148543); 
        \end{tikzpicture}    &
        \begin{tikzpicture} 
            \node[circle, fill, inner sep=0pt, minimum size = 1mm] (383617511) at (0.0,0) {}; 
            \node (383617512) at (0.333,0) {$-$}; 
            \node[circle, fill, inner sep=0pt, minimum size = 1mm] (383617513) at (0.667,0) {}; 
            \draw (383617511) to[bend left=65] (383617513); 
        \end{tikzpicture}  \\
        $\scriptstyle (1^+)(2^-)(3^-)$ & $\scriptstyle (1^-)(2^+)(3^-)$ & $\scriptstyle (1^-)(2^-)(3^+)$ & $\scriptstyle (1\,2)(3^-)$ & $\scriptstyle (1^-)(2\,3)$ & $\scriptstyle (1\,3)(2^-)$\\
        $\scriptstyle 1^+2^-3^-$ & $\scriptstyle 1^-2^+3^-$ & $\scriptstyle 1^-2^-3^+$ & $\scriptstyle 213^-$ & $\scriptstyle 1^-32$ & $\scriptstyle 32^-1$
    \end{tabular}
\end{center}
where the last two lines are cycle notation and one-line notation, respectively.  When a clan consists entirely of fixed points, we simplify the one-line notation: $--+$ instead of $1^-2^-3^+$.
\end{ex}

Our treatment of clans will be combinatorial and algebraic, but their origins are in geometry. A \emph{complete flag} $F_{\bullet}$ in a vector space $V$ is a chain of subspaces $F_1 \subseteq F_2 \subseteq \cdots \subseteq F_n = V$ where $\dim F_i = i$. Let $\Fl(V)$ denote the set of complete flags in $V$. The (left) action of $\GL(V)$ on $V$ induces an action on $\Fl(V)$, hence an action of any subgroup of $\GL(V)$ on $\Fl(V)$. Identify $\GL(\CC^p) \times \GL(\CC^q)$ with the subgroup of $\GL(\CC^{p+q})$ consisting of block diagonal matrices with a $p \times p$ block in the upper left and a $q \times q$ block in the lower right. Then the orbits of $\GL(\CC^p) \times \GL(\CC^q)$ on $\Fl(\CC^{p+q})$ are in bijection with the $(p,q)$-clans in a natural way \cite{matsuki-oshima-clans, wyser-clans}.

A closed subgroup $K \subseteq \GL(\CC^n)$ is \emph{spherical} if it acts on $\Fl(\CC^n)$ with finitely many orbits (more generally, one can replace $\GL(\CC^n)$ with a reductive algebraic group $G$ and $\Fl(\CC^n)$ with the generalized flag variety of $G$). From the geometry arises a natural partial order on the set of $K$-orbits called \emph{weak order} \cite{richardson-springer}. This poset is graded by codimension and has a unique minimal element. The central objects of this paper are the saturated chains containing the minimal element in the case $K = \GL(\CC^p) \times \GL(\CC^q)$.

The covering relations in weak order are labelled by integers in $[n-1]$, so a saturated chain from the minimal element to a clan $\gamma$ can be identified with a word on the alphabet $[n-1]$, and we call such a word a \emph{reduced word} for $\gamma$. This is by analogy with the more familiar case where $K$ is the subgroup of lower triangular matrices, in which the $K$-orbits on $\Fl(\CC^n)$ are in bijection with permutations of $n$, and their closures are the Schubert varieties in $\Fl(\CC^n)$. There, weak order is defined by the covering relations $ws_i < w$ whenever $ws_i$ has fewer inversions than $w$, where $s_i$ is the adjacent transposition $(i\,\,i{+}1) \in S_n$. The saturated chains from the minimal element to $w$ are then labeled by the \emph{reduced words} of $w$: the minimal-length words $a_1 \cdots a_\ell$ such that $w = s_1 \cdots s_{a_{\ell}}$.

\begin{ex} \label{ex:weak-order} Here is the weak order on $\Clan_{1,2}$ (we have labelled the edges by the adjacent transpositions $s_1, \ldots, s_{n-1}$ rather than the integers $1, \ldots, n-1$):
    
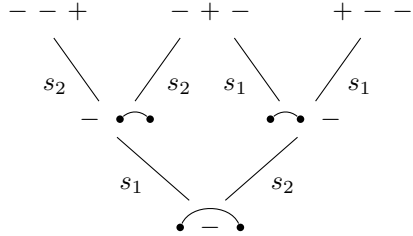
\begin{center}
        \begin{tikzpicture}[auto, scale=1.2]
            \node[circle, fill, inner sep=0pt, minimum size = 1mm] (108353022) at (0.0,0) {}; 
            \node (108353023) at (0.333,0) {$-$}; 
            \node[circle, fill, inner sep=0pt, minimum size = 1mm] (108353024) at (0.667,0) {}; 
            \draw (108353022) to[bend left=65] (108353024); 
            \node (9417173) at (-1.0,1.2) {$-$}; 
            \node[circle, fill, inner sep=0pt, minimum size = 1mm] (9417174) at (-0.667,1.2) {}; 
            \node[circle, fill, inner sep=0pt, minimum size = 1mm] (9417175) at (-0.333,1.2) {}; 
            \draw (9417174) to[bend left=40] (9417175); 
            \node[circle, fill, inner sep=0pt, minimum size = 1mm] (940972933) at (1.0,1.2) {}; 
            \node[circle, fill, inner sep=0pt, minimum size = 1mm] (940972934) at (1.333,1.2) {}; 
            \node (940972935) at (1.667,1.2) {$-$}; 
            \draw (940972933) to[bend left=40] (940972934); 
            \node (90526312) at (-1.8,2.4) {$-$}; 
            \node (90526313) at (-1.467,2.4) {$-$}; 
            \node (90526314) at (-1.133,2.4) {$+$}; 
            \node (531055366) at (0.0,2.4) {$-$}; 
            \node (531055367) at (0.333,2.4) {$+$}; 
            \node (531055368) at (0.667,2.4) {$-$}; 
            \node (324345687) at (1.8,2.4) {$+$}; 
            \node (324345688) at (2.133,2.4) {$-$}; 
            \node (324345689) at (2.467,2.4) {$-$}; 

            \draw (.1,.3) to node {$s_1$} (-0.7,1);
            \draw (.5,.3) to[swap] node {$s_2$} (1.3,1);
            \draw (-.9,1.4) to node {$s_2$} (-1.4, 2.1);
            \draw (-.5,1.4) to[swap] node {$s_2$} (0, 2.1);
            \draw (1.1,1.4) to node {$s_1$} (0.6,2.1);
            \draw (1.5,1.4) to[swap] node {$s_1$} (2,2.1);
        \end{tikzpicture} 
    \end{center}
    The reduced words of $-+-$ are $\bf 12$ and $\bf 21$, while the only reduced word of $--+$ is $\bf 12$.
\end{ex}

In $S_n$ (or in any Coxeter group), one can obtain any reduced word for $w$ from any other via simple transformations. Let $\Red(w)$ be the set of reduced words of $w \in S_n$.
\begin{thm}[Matsumoto-Tits]  \label{thm:matsumoto-tits} Let $\equiv$ be the equivalence relation on the set of words on the alphabet $[n-1]$ defined as the transitive closure of the relations
    \begin{align*}
        \cdots ik \cdots \equiv \cdots ki \cdots \qquad &\text{if $|i-k| > 1$}\\
        \cdots iji \cdots \equiv \cdots jij \cdots \qquad &\text{if $|i-j| = 1$}.
    \end{align*}
    Every equivalence class of $\equiv$ either contains no reduced word for any $w \in S_n$, or consists entirely of reduced words. Moreover, the classes containing reduced words are exactly the sets $\Red(w)$ for $w \in S_n$.
\end{thm}

Example~\ref{ex:weak-order} shows that an exact analogue of this theorem cannot hold for clans, because different clans can share the same reduced word. However, we do get a similar result by relaxing the constraint that each reduced word set must be a single equivalence class. Let $\Red(\gamma)$ be the set of reduced words for a clan $\gamma$, and let $\Clan_{p,q}$ be the set of $(p,q)$-clans.

\begin{thm} \label{thm:clan-matsumoto-tits} Let $\equiv$ be the equivalence relation on the set of words on $[n-1]$ defined as the transitive closure of the relations $a_1 a_2 \cdots a_{\ell} \equiv (n-a_1)a_2 \cdots a_{\ell}$ together with the Coxeter relations of Theorem~\ref{thm:matsumoto-tits}. Every equivalence class of $\equiv$ either contains no reduced word for any $\gamma \in \Clan_{p,q}$, or consists entirely of reduced words. Also, when restricted to reduced words, $\equiv$ is the strongest equivalence relation for which each $\Red(\gamma)$ is a union of equivalence classes. In other words, $a \equiv b$ if and only if
    \begin{equation*} 
        \{\gamma \in \Clan_{p,q} : a \in \Red(\gamma)\} = \{\gamma \in \Clan_{p,q} : b \in \Red(\gamma)\}.
    \end{equation*}
\end{thm}

In \cite{stanleysymm}, Stanley defined a symmetric function $F_w$ associated to a permutation $w$ in which the coefficient of a squarefree monomial is the number of reduced words of $w$. For many $w$ of interest (e.g. the reverse permutation $n\cdots 21$), the Schur expansion of $F_w$ is simple enough that one obtains enumerations of reduced words in terms of standard tableaux. A formula of Billey-Jockusch-Stanley \cite{billeyjockuschstanley} shows that $F_w$ is a certain limit of Schubert polynomials, which represent the cohomology classes of Schubert varieties in $\Fl(\CC^n)$.

We follow a similar approach to prove some enumerative results for reduced words of clans in Section~\ref{sec:enum}. Wyser and Yong \cite{wyser-yong-clans} define polynomials which represent the cohomology classes of the $\GL(\CC^p) \times \GL(\CC^q)$-orbit closures on $\Fl(\CC^n)$, and a result of Brion \cite{brion} implies an analogue of the Billey-Jockusch-Stanley formula. We define the Stanley symmetric function $F_\gamma$ of a clan $\gamma$ as a limit of the Wyser-Yong polynomials. In particular, the maximal clans in weak order are the \emph{matchless} clans, those whose underlying involution is the identity permutation, and we show in this case that $F_\gamma$ is the product of two Schur polynomials. This gives a simple product formula for the number of reduced words:

\begin{thm} \label{thm:enum-formula} Suppose $\gamma \in \Clan_{p,q}$ is matchless with $+$'s in positions $\phi^+ \subseteq [n]$ and $-$'s in positions $\phi^- = [n] \setminus \phi^+$. Then
    \begin{equation*}
        \#\Red(\gamma) = (pq)! \prod_{\substack{i \in \phi^+ \\ j \in \phi^-}} \frac{1}{|i-j|}.
    \end{equation*}
\end{thm}

The permutation $w \in S_n$ with the most reduced words is the reverse permutation $n\cdots 21$, the unique maximal element in weak order. Similarly, a clan $\gamma \in \Clan_{p,q}$ maximizing $\#\Red(\gamma)$ must be matchless, but otherwise it is not obvious what these clans are. We investigate this question in Section~\ref{sec:max}, including connections to work of Pittel and Romik on random Young tableaux of rectangular shape \cite{pittel-romik} suggested by Theorem~\ref{thm:enum-formula}.

The orbits of the orthogonal group $\operatorname{O}(\CC^n)$ on $\Fl(\CC^n)$ are indexed by the involutions in $S_n$, and the resulting weak order on involutions has been studied by various authors \cite{CJW, HMP1, HMP2, hultman-twisted-involutions, richardson-springer}. If one forgets the signs of fixed points, clan weak order becomes the opposite of involution weak order, and this relationship is explored in Section~\ref{sec:inv-vs-clan}. In particular, from known results on reduced words in involution weak order \cite{HMP4} we deduce another enumeration:
\begin{thm}
    The number of maximal chains in $\Clan_{p,q}$ is 
    \begin{equation*}
        2^{pq} {pq \choose \lambda} \prod_{i=1}^{\min(p,q)} {p+q-2i \choose p-i, q-i}^{-1},
    \end{equation*}
    where $\lambda = (p+q-1,p+q-3,\ldots,p-q+1)$ and ${pq \choose \lambda}$ is the multinomial coefficient ${pq \choose \lambda_1, \ldots, \lambda_\ell}$. This is $2^q$ times the number of marked shifted standard tableaux of shifted shape $\lambda$ (cf. Definition~\ref{def:shifted}).
\end{thm}

\subsection*{Acknowledgements} 
We thank Eric Marberg for useful suggestions, including the question motivating Section~\ref{sec:max}, and Zach Hamaker for pointing out the relevance of the work of Romik and Pittel. This work was done as part of the University of Michigan REU program, and we thank David Speyer and everyone else involved in the program.

\section{Reduced words for clans}

Let $\Clan_{p,q}$ be the set of $(p,q)$-clans. We usually write $n$ to mean $p+q$ without comment. Let $s_i$ be the adjacent transposition $(i\,\,i{+}1)$, and write $\iota(\gamma)$ for the underlying involution of a clan $\gamma$. We define conjugation of $\gamma$ by $s_i$ as follows: take the underlying involution of $s_i \gamma s_i$ to be $s_i \iota(\gamma) s_i$, and give the fixed points of $s_i\gamma s_i$ the same signs that they have in $\gamma$ except that the signs of $i$ and $i+1$ (if any) become the respective signs of $i+1$ and $i$.

\begin{ex}
\begin{gather*}
s_2 (1\,2)(3^-) s_2 = (1\,3)(2^-) \quad \text{and} \quad  s_1 (1\,2)(3^-) s_1 = (1\,2)(3^-);\\
s_2 (1^-)(2^-)(3^+) s_2 = (1^-)(2^+)(3^-).
\end{gather*}
\end{ex}
Conjugation preserves the number of $+$'s and $-$'s, hence the set of $(p,q)$-clans. Imagining a clan as an ordered row of unlabeled nodes, each of which has a strand or a sign attached to it (as in Example~\ref{ex:clan-pictures}), conjugation by $s_i$ simply swaps the $i$\textsuperscript{th} and $(i+1)$\textsuperscript{th} node, with any attached strand or sign being carried along.

Using conjugation we now define a different, partial action of the $s_i$ on clans. 
\begin{itemize}
    \item If $i$ and $i+1$ are fixed points of $\gamma$ of opposite sign, then $\gamma \ast s_i$ is $\gamma$ except that $i$ and $i+1$ are now matched.
    \item If $i$ and $i+1$ are matched in $\gamma$, or are fixed points of equal sign, we leave $\gamma \ast s_i$ undefined.
    \item If $i$ and $i+1$ are not fixed points and are not matched with each other, $\gamma \ast s_i = s_i \gamma s_i$.
\end{itemize}
The operation $\gamma \mapsto \gamma \ast s_i$ is not invertible for the same reason that we leave $\gamma \ast s_i$ undefined in the second case: if $i$ and $i+1$ are matched in $\gamma$ then they ought to be replaced in $\gamma \ast s_i$ by $+-$ or $-+$, but there is no reason to choose one over the other.

Let $\ell(w)$ be the Coxeter length of a permutation $w$ (number of inversions).
\begin{defn} The \emph{weak order} on $\Clan_{p,q}$ is the transitive closure of the relation $\gamma \ast s_i < \gamma$ if $\ell(\iota(\gamma \ast s_i)) > \ell(\iota(\gamma))$. \end{defn}
One should mark the reversal here compared to weak Bruhat order on the symmetric group, which has covering relations $ws_i < w$ whenever $\ell(ws_i) < \ell(w)$. By contrast, the largest elements of $\Clan_{p,q}$ have the \emph{fewest} inversions when viewed as permutations.

\begin{defn}
A clan $\gamma$ is \emph{matchless} if $\iota(\gamma)$ is the identity permutation. 
\end{defn}
There are ${p+q \choose p,q}$ matchless clans in $\Clan_{p,q}$, and they are exactly the maximal elements in weak order. There is a unique minimal element in weak order on $\Clan_{p,q}$, which we will call $\gamma_{p,q}$: its underlying involution is $(1\,\,n)(2\,\,n{-}1)\cdots (m\,\,n{-}m{+}1)$ where $m = \min(p,q)$, and the fixed points $m+1, m+2, \ldots, n-m$ are are all labeled with the sign of $p-q$.
\begin{ex}
    The minimal element $\gamma_{5,3} \in \Clan_{5,3}$ has $|p-q| = 2$ fixed points, labeled $+$ since $p-q > 0$, and $\min(p,q) = 3$ arcs:
    
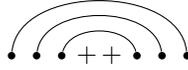
\begin{center}
        \begin{tikzpicture}
            \node[circle, fill, inner sep=0pt, minimum size = 1mm] (403475827) at (0.0,0) {}; 
            \node[circle, fill, inner sep=0pt, minimum size = 1mm] (403475828) at (0.333,0) {}; 
            \node[circle, fill, inner sep=0pt, minimum size = 1mm] (403475829) at (0.667,0) {}; 
            \node (403475830) at (1.0,0) {$+$}; 
            \node (403475831) at (1.333,0) {$+$}; 
            \node[circle, fill, inner sep=0pt, minimum size = 1mm] (403475832) at (1.667,0) {}; 
            \node[circle, fill, inner sep=0pt, minimum size = 1mm] (403475833) at (2.0,0) {}; 
            \node[circle, fill, inner sep=0pt, minimum size = 1mm] (403475834) at (2.333,0) {}; 
            \draw (403475827) to[bend left=84] (403475834);\draw (403475828) to[bend left=82] (403475833);\draw (403475829) to[bend left=75] (403475832); 
        \end{tikzpicture}
    \end{center}
\end{ex}

\begin{defn}
A word $a_1 \cdots a_{\ell}$ with letters in $\NN$ is a \emph{reduced word} for $\gamma \in \Clan_{p,q}$ if there is a saturated chain from the minimal element $\gamma_{p,q} \in \Clan_{p,q}$ to $\gamma$ with edge labels $s_{a_1}, \ldots, s_{a_{\ell}}$ (in that order, beginning with $\gamma_{p,q}$ and ending with $\gamma$). Let $\clRed(\gamma)$ be the set of reduced words of $\gamma$. Similarly, $\Red(w)$ denotes the set of reduced words of a permutation $w \in S_n$.
\end{defn} 
We will use bold for reduced words to distinguish them from permutations.

\begin{ex}\label{ex:reduced-words} From Example~\ref{ex:weak-order} one can see that
    \begin{align*} &\clRed(\!\raisebox{-2mm}{
        \begin{tikzpicture}
            \node[circle, fill, inner sep=0pt, minimum size = 1mm] (836675256) at (0.0,0) {}; 
            \node (836675257) at (0.333,0) {$-$}; 
            \node[circle, fill, inner sep=0pt, minimum size = 1mm] (836675258) at (0.667,0) {}; 
            \draw (836675256) to[bend left=65] (836675258); 
        \end{tikzpicture}}\,)  = \{\epsilon\}\\
        &\clRed(\!\!\!\raisebox{-2mm}{
            \begin{tikzpicture}
                \node (268222532) at (0.0,0) {$-$}; 
                \node[circle, fill, inner sep=0pt, minimum size = 1mm] (268222533) at (0.333,0) {}; 
                \node[circle, fill, inner sep=0pt, minimum size = 1mm] (268222534) at (0.667,0) {}; 
                \draw (268222533) to[bend left=40] (268222534); 
            \end{tikzpicture}}\,)  = \{\bf 1\}\\
        &\clRed(\!\raisebox{-2mm}{
            \begin{tikzpicture}
                \node[circle, fill, inner sep=0pt, minimum size = 1mm] (291777389) at (0.0,0) {}; 
                \node[circle, fill, inner sep=0pt, minimum size = 1mm] (291777390) at (0.333,0) {}; 
                \node (291777391) at (0.667,0) {$-$}; 
                \draw (291777389) to[bend left=40] (291777390); 
            \end{tikzpicture}}\!)  = \{\bf 2\}\\
        &\clRed(--+)  = \{\bf 12\}\\
        &\clRed(-+-) = \{{\bf 12}, \mathbf{21}\}\\
        &\clRed(+--) = \{\bf 21\}
    \end{align*}
    where $\epsilon$ is the empty word. Unlike reduced words in Coxeter groups, a word can be a reduced word for more than one clan.
\end{ex}

\begin{warn} We write reduced words starting at the minimal element $\gamma_{p,q} \in \Clan_{p,q}$ by analogy with reduced words for Coxeter groups. However, if $a_1 \cdots a_{\ell}$ is a reduced word for $\gamma \in \Clan_{p,q}$, then $(\cdots ((\gamma_{p,q} \ast s_{a_1}) \ast s_{a_2}) \ast \cdots )\ast s_{a_\ell}$ need not be defined (although if it is, then it equals $\gamma$). Rather, one must say that $a_1 \cdots a_{\ell}$ is a reduced word for $\gamma$ if $(\cdots ((\gamma \ast s_{a_\ell}) \ast s_{a_{\ell-1}}) \ast \cdots )\ast s_{a_1} = \gamma_{p,q}$ and $\ell$ is minimal.
\end{warn}

\begin{rem} \label{rem:geom-weak-order} The motivation for this definition of weak order on clans comes from geometry. Given any subset $Y \subseteq \Fl(\CC^{n})$ and $1 \leq i < n$, let $Y \ast s_i$ be the subset
    \begin{equation*}
        \{F_{\bullet} : F_1 \subseteq \cdots \subseteq F_{i-1} \subseteq F' \subseteq  F_{i+1} \subseteq \cdots \subseteq F_{n} \text{ is in $Y$ for some $i$-dimensional $F'$}\}.
    \end{equation*}
    In particular, $Y \ast s_i$ contains $Y$. Recall from the introduction that the $\GL(\CC^p) \times \GL(\CC^q)$-orbits on $\Fl(\CC^n)$ can be labeled by $(p,q)$-clans. Letting $Y_{\gamma}$ denote the orbit labeled by $\gamma$, one has $Y_{\gamma} \ast s_i = Y_{\gamma \ast s_i}$ if $\gamma \ast s_i < \gamma$. This operation is important in Schubert calculus: the Zariski closures $\overline{Y}_\gamma$ and $\overline{Y}_{\gamma \ast s_i}$ have associated cohomology classes $[\overline{Y}_\gamma]$ and $[\overline{Y}_{\gamma \ast s_i}]$, and under the Borel isomorphism identifying the cohomology ring $H^*(\Fl(\CC^{n}), \ZZ)$ with a quotient of $\ZZ[x_1, \ldots, x_{n}]$, these two classes are related by a divided difference operator; see Section~\ref{sec:enum}. (This can all be phrased more generally for a complex reductive group $G$ and simple generator $s$ of its Weyl group $W$: in passing from $Y$ to $Y \ast s$, we are first projecting $Y$ from the flag variety of $G$ onto the partial flag variety associated to the parabolic subgroup $\langle s \rangle$ of $W$, and then applying the inverse image of the same projection.)
\end{rem} 

When passing from $\gamma$ to $\gamma \ast s_i < \gamma$, only the $i$\textsuperscript{th} and $(i+1)$\textsuperscript{th} nodes in the matching diagrams change, and it is helpful to have a list of the possible local moves. In Figure~\ref{fig:local-moves}, we have drawn the $i$\textsuperscript{th} and $(i+1)$\textsuperscript{th} nodes of $\gamma$ on the left, and those of $\gamma \ast s_i$ on the right, assuming $\gamma \ast s_i < \gamma$.
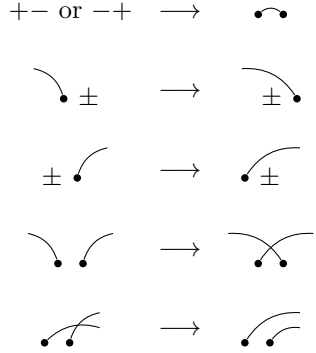
\begin{figure}  \renewcommand{\arraystretch}{2.5} \caption{Possible local changes in a covering relation in weak order} \label{fig:local-moves} 
\begin{tabular}{ccc}   
    $+-$ or $-+$ & $\longrightarrow$ & 
    \begin{tikzpicture}
        \node[circle, fill, inner sep=0pt, minimum size = 1mm] (632987728) at (0.0,0) {}; 
        \node[circle, fill, inner sep=0pt, minimum size = 1mm] (632987729) at (0.333,0) {}; 
        \draw (632987728) to[bend left=40] (632987729); 
    \end{tikzpicture}\\
    
    \raisebox{-2.5mm}{\begin{tikzpicture}
        \node[circle, fill, inner sep=0pt, minimum size = 1mm] (1) at (0.0,0) {};
        \node (2) at (0.333,0) {$\pm$};
        \draw (1) to[bend right] (-.4,.4);
    \end{tikzpicture}} & $\longrightarrow$ &
    \raisebox{-2.5mm}{\begin{tikzpicture}
        \node[circle, fill, inner sep=0pt, minimum size = 1mm] (2) at (0.333,0) {};
        \node (1) at (0,0) {$\pm$};
        \draw (2) to[bend right] (-.4,.4);
    \end{tikzpicture}}\\

    \raisebox{-2.5mm}{\begin{tikzpicture}
        \node[circle, fill, inner sep=0pt, minimum size = 1mm] (2) at (0.333,0) {};
        \node (1) at (0,0) {$\pm$};
        \draw (2) to[bend left] (0.733,.4);
    \end{tikzpicture}} & $\longrightarrow$ &
    \raisebox{-2.5mm}{\begin{tikzpicture}
        \node[circle, fill, inner sep=0pt, minimum size = 1mm] (1) at (0,0) {};
        \node (2) at (0.333,0) {$\pm$};
        \draw (1) to[bend left] (0.733,.4);
    \end{tikzpicture}}\\

    \raisebox{-1.5mm}{\begin{tikzpicture}
        \node[circle, fill, inner sep=0pt, minimum size = 1mm] (1) at (0,0) {};
        \node[circle, fill, inner sep=0pt, minimum size = 1mm] (2) at (0.333,0) {};
        \draw (1) to[bend right] (-0.4,0.4);
        \draw (2) to[bend left] (0.733,0.4);
    \end{tikzpicture}} & $\longrightarrow$ &
    \raisebox{-1.5mm}{\begin{tikzpicture}
        \node[circle, fill, inner sep=0pt, minimum size = 1mm] (1) at (0,0) {};
        \node[circle, fill, inner sep=0pt, minimum size = 1mm] (2) at (0.333,0) {};
        \draw (2) to[bend right] (-0.4,0.4);
        \draw (1) to[bend left] (0.733,0.4);
    \end{tikzpicture}}\\

    \raisebox{-1.5mm}{\begin{tikzpicture}
        \node[circle, fill, inner sep=0pt, minimum size = 1mm] (1) at (0,0) {};
        \node[circle, fill, inner sep=0pt, minimum size = 1mm] (2) at (0.333,0) {};
        \draw (1) to[bend left] (0.733,0.2);
        \draw (2) to[bend left] (0.733,0.4);
    \end{tikzpicture}} & $\longrightarrow$ &
    \raisebox{-1.5mm}{\begin{tikzpicture}
        \node[circle, fill, inner sep=0pt, minimum size = 1mm] (1) at (0,0) {};
        \node[circle, fill, inner sep=0pt, minimum size = 1mm] (2) at (0.333,0) {};
        \draw (2) to[bend left] (0.733,0.2);
        \draw (1) to[bend left] (0.733,0.4);
    \end{tikzpicture}}\\
\end{tabular}
\end{figure}

\begin{lem}[\cite{richardson-springer}, Lemma 3.16] \label{lem:braid-relation-closure} The reduced word set $\clRed(\gamma)$ of any $\gamma \in \Clan_{p,q}$ is closed under the Coxeter relations for $S_{n}$ (cf. Theorem~\ref{thm:matsumoto-tits}). That is, $\clRed(\gamma)$ is closed under the following operations on words:
    \begin{align*}
        \cdots ik \cdots \leadsto \cdots ki \cdots \qquad &\text{if $|i-k| > 1$}\\
        \cdots iji \cdots \leadsto \cdots jij \cdots \qquad &\text{if $|i-j| = 1$}.
    \end{align*}
    Moreover, any $a \in \Red(\gamma)$ is a reduced word for some permutation.
\end{lem}

\begin{defn} \label{defn:atoms} The set of \emph{atoms} of a $(p,q)$-clan $\gamma$ is the set of permutations $\A(\gamma) \subseteq S_n$ such that $\clRed(\gamma) = \bigcup_{w \in \A(\gamma)} \Red(w)$, guaranteed to exist by Lemma~\ref{lem:braid-relation-closure}. \end{defn}

    \begin{ex} \label{ex:atoms} Example~\ref{ex:reduced-words} shows that $\A(--+) = \{s_1 s_2\} = \{231\}$ and $\A(-+-) = \{s_1 s_2, s_2 s_1\} = \{231, 312\}$. A more interesting example: $\A(+--+) = \{4132, 3241\}$, because
        \begin{align*}
            \clRed(+--+) &= \{\bf 2321, 3231, 3213, 1231, 1213, 2123\}\\
            & = \{\bf 2321, 3231, 3213\} \cup \{\bf 1231, 1213, 2123\}\\
            & = \Red(4132) \cup \Red(3241).
        \end{align*}
    \end{ex}

    Given a word $a$, let $\Gamma(a)$ be the set of clans $\gamma \in \Clan_{p,q}$ such that $a \in \Red(\gamma)$. If $w$ is a permutation, we also write $\Gamma(w)$ for the set of clans $\gamma \in \Clan_{p,q}$ such that $w \in \A(\gamma)$.

\begin{defn} Let $\sim$ be the strongest equivalence relation on reduced words for members of $\Clan_{p,q}$ with the property that each $\Red(\gamma)$ for $\gamma \in \Clan_{p,q}$ is a union of equivalence classes. Equivalently, $a \sim b$ if and only if $\Gamma(a) = \Gamma(b)$. \end{defn}

Lemma~\ref{lem:braid-relation-closure} shows that $\sim$ respects the Coxeter relations in the sense that if $a, b \in \Red(w)$ for some $w \in S_n$, then $a \sim b$. We can therefore simplify the problem of describing $\sim$ by ``factoring out'' these relations. If $v, w$ are atoms for some members of $\Clan_{p,q}$, write $v \sim w$ if $\Gamma(v) = \Gamma(w)$. Then $a \sim b$ if and only if $a \in \Red(v), b \in \Red(w)$ for some atoms $v \sim w$.

To understand this equivalence relation on permutations we need a better understanding of the sets $\A(\gamma)$. Given a subset $S \subseteq [n]$ and a clan $\gamma \in \Clan_{p,q}$, call a pair $(i < j) \in S$ \emph{valid} if either $i$ and $j$ are matched by $\gamma$, or if they are fixed points of opposite sign which are adjacent in the sense that there is no $i' \in S$ with $i < i' < j$. Consider the following algorithm which (nondeterministically) builds a permutation $w \in S_n$ by removing one pair of points from $[n]$ at a time and correspondingly deciding upon two entries of $w$. Set $S := [n]$ to start. 
\begin{alg} \label{alg:atoms} \hfill
\begin{enumerate}[(a)]
    \item Choose a valid pair $(i < j) \in S$ such that $\gamma$ has no \emph{matched} pair $i',j' \in S$ with $i' < i < j < j'$. 
    \begin{itemize}
        \item If $i < j$ are matched by $\gamma$, set $w(i) = s+1$ and $w(j) = n-s$, where $s = (n-|S|)/2$ (this is the number of pairs deleted from $[n]$ in step (c) so far).
        \item If $i < j$ are fixed by $\gamma$, set $w(i) = n-s$ and $w(j) = s+1$, with $s$ as above.
    \end{itemize}
    \item If $S$ consists entirely of fixed points of $\gamma$ of the same sign, fill in the remaining undefined entries of $w$ with the unused entries of $[n]$ in increasing order, and return $w$.
    \item If we did not finish in step (b), then replace $S$ with $S \setminus \{i,j\}$ and go back to (a).
\end{enumerate}
\end{alg}

\begin{ex} \label{ex:atom-alg}
    Here is one way this algorithm can run when $\gamma = (1\,9)(2^+)(3^+)(4\,7)(5^{-})(6\,8)$:
    \begin{center} \setlength{\tabcolsep}{2pt}
    \begin{tabular}{cccccccccc}
        \begin{tikzpicture}[scale=0.8]
            \node[circle, fill, inner sep=0pt, minimum size = 1mm] (654424181) at (0.0,0) {}; 
            \node (654424182) at (0.333,0) {$+$}; 
            \node (654424183) at (0.667,0) {$+$}; 
            \node[circle, fill, inner sep=0pt, minimum size = 1mm] (654424184) at (1.0,0) {}; 
            \node (654424185) at (1.333,0) {$-$}; 
            \node[circle, fill, inner sep=0pt, minimum size = 1mm] (654424186) at (1.667,0) {}; 
            \node[circle, fill, inner sep=0pt, minimum size = 1mm] (654424187) at (2.0,0) {}; 
            \node[circle, fill, inner sep=0pt, minimum size = 1mm] (654424188) at (2.333,0) {}; 
            \node[circle, fill, inner sep=0pt, minimum size = 1mm] (654424189) at (2.667,0) {}; 
            \draw (654424181) to[bend left=85] (654424189);\draw (654424184) to[bend left=75] (654424187);\draw (654424186) to[bend left=65] (654424188); 
            \node at (0,-.3) {$\scriptstyle 1$}; \node at (0.333,-.3) {$\scriptstyle 2$};\node at (0.666,-.3) {$\scriptstyle 3$};
            \node at (1,-.3) {$\scriptstyle 4$}; \node at (1.333,-.3) {$\scriptstyle 5$};\node at (1.666,-.3) {$\scriptstyle 6$};
            \node at (2,-.3) {$\scriptstyle 7$}; \node at (2.333,-.3) {$\scriptstyle 8$};\node at (2.666,-.3) {$\scriptstyle 9$};
        \end{tikzpicture} & \raisebox{3mm}{$\to$} &
        \begin{tikzpicture}[scale=0.8]
            \node (435374291) at (0.333,0) {$+$};
            \node (435374292) at (0.667,0) {$+$};
            \node[circle, fill, inner sep=0pt, minimum size = 1mm] (435374293) at (1.0,0) {};
            \node (435374294) at (1.333,0) {$-$};
            \node[circle, fill, inner sep=0pt, minimum size = 1mm] (435374295) at (1.667,0) {};
            \node[circle, fill, inner sep=0pt, minimum size = 1mm] (435374296) at (2.0,0) {};
            \node[circle, fill, inner sep=0pt, minimum size = 1mm] (435374297) at (2.333,0) {};
            \node at (0,-.3) {$\scriptstyle 1$}; \node at (0.333,-.3) {$\scriptstyle 2$};\node at (0.666,-.3) {$\scriptstyle 3$};
            \node at (1,-.3) {$\scriptstyle 4$}; \node at (1.333,-.3) {$\scriptstyle 5$};\node at (1.666,-.3) {$\scriptstyle 6$};
            \node at (2,-.3) {$\scriptstyle 7$}; \node at (2.333,-.3) {$\scriptstyle 8$};\node at (2.666,-.3) {$\scriptstyle 9$};
            \draw (435374293) to[bend left=75] (435374296);\draw (435374295) to[bend left=65] (435374297);
        \end{tikzpicture} & \raisebox{3mm}{$\to$} &
        \begin{tikzpicture}[scale=0.8]
            \node (435374291) at (0.333,0) {$+$};
            \node (435374292) at (0.667,0) {$+$};
            \node (435374294) at (1.333,0) {$-$};
            \node[circle, fill, inner sep=0pt, minimum size = 1mm] (435374295) at (1.667,0) {};
            \node[circle, fill, inner sep=0pt, minimum size = 1mm] (435374297) at (2.333,0) {};
            \node at (0,-.3) {$\scriptstyle 1$}; \node at (0.333,-.3) {$\scriptstyle 2$};\node at (0.666,-.3) {$\scriptstyle 3$};
            \node at (1,-.3) {$\scriptstyle 4$}; \node at (1.333,-.3) {$\scriptstyle 5$};\node at (1.666,-.3) {$\scriptstyle 6$};
            \node at (2,-.3) {$\scriptstyle 7$}; \node at (2.333,-.3) {$\scriptstyle 8$};\node at (2.666,-.3) {$\scriptstyle 9$};
            \draw (435374295) to[bend left=65] (435374297);
        \end{tikzpicture} & \raisebox{3mm}{$\to$} &
        \begin{tikzpicture}[scale=0.8]
            \node (435374291) at (0.333,0) {$+$};
            \node[circle, fill, inner sep=0pt, minimum size = 1mm] (435374295) at (1.667,0) {};
            \node[circle, fill, inner sep=0pt, minimum size = 1mm] (435374297) at (2.333,0) {};
            \node at (0,-.3) {$\scriptstyle 1$}; \node at (0.333,-.3) {$\scriptstyle 2$};\node at (0.666,-.3) {$\scriptstyle 3$};
            \node at (1,-.3) {$\scriptstyle 4$}; \node at (1.333,-.3) {$\scriptstyle 5$};\node at (1.666,-.3) {$\scriptstyle 6$};
            \node at (2,-.3) {$\scriptstyle 7$}; \node at (2.333,-.3) {$\scriptstyle 8$};\node at (2.666,-.3) {$\scriptstyle 9$};
            \draw (435374295) to[bend left=65] (435374297);
        \end{tikzpicture} & \raisebox{3mm}{$\to$} &
        \begin{tikzpicture}[scale=0.8]
            \node (435374291) at (0.333,0) {$+$};
            \node at (0,-.3) {$\scriptstyle 1$}; \node at (0.333,-.3) {$\scriptstyle 2$};\node at (0.666,-.3) {$\scriptstyle 3$};
            \node at (1,-.3) {$\scriptstyle 4$}; \node at (1.333,-.3) {$\scriptstyle 5$};\node at (1.666,-.3) {$\scriptstyle 6$};
            \node at (2,-.3) {$\scriptstyle 7$}; \node at (2.333,-.3) {$\scriptstyle 8$};\node at (2.666,-.3) {$\scriptstyle 9$};
        \end{tikzpicture} \\
        $w = \_\,\_\,\_\,\_\,\_\,\_\,\_\,\_\,\_$ & & $w = 1\_\,\_\,\_\,\_\,\_\,\_\,\_9$ & & $w = 1\_\,\_2\_\,\_\,8\_9$ & & $w = 1\_723\_8\_9$ & & $w = 1\_\,7236849$ &
    \end{tabular}
\end{center}
At this point no more pairs can be selected in step (a), so the algorithm returns $157236849$.
\end{ex}

\begin{thm}[\cite{CJW}] \label{thm:cjw} $\A(\gamma)$ is the set of permutations which can be generated by Algorithm~\ref{alg:atoms}. \end{thm}
    We note that \cite{CJW} works with the set $\mathcal{W}(\gamma) := \{w^{-1} : w \in \A(\gamma)\}$ rather than our $\A(\gamma)$.

The possible outcomes of Algorithm~\ref{alg:atoms} can be also encoded by recording, for each $i$ which is removed in the course of the algorithm, which step it was removed at.
\begin{defn}
    A \emph{labelled shape} for $\gamma$ is a pair $(\omega, F)$ where $\omega$ is the partial function $[n] \dashrightarrow \NN$ obtained from an instance of Algorithm~\ref{alg:atoms} by setting $\omega(i) = \omega(j) = k$ if $\{i,j\}$ is the $k$\th pair deleted from $[n]$ in step (c) of the algorithm, and $F$ is the subset of the domain of $\omega$ consisting of fixed points of $\gamma$.
\end{defn}
We think of a labelled shape $(\omega,F)$ as the edge-labelled partial matching on $[n]$ with an arc labeled $k$ matching $i$ and $j$ for each $\omega^{-1}(k) = \{i,j\}$, where the arc is marked if $i,j \in F$. We draw these marked arcs as doubled edges. Given this marking, we will omit $F$ from the notation since it can be recovered as the set of endpoints of the marked arcs.

\begin{ex} \label{ex:labelled-shape} The instance of Algorithm~\ref{alg:atoms} in Example~\ref{ex:atom-alg} gives the labelled shape
    
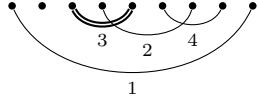
\begin{center} \vspace{2mm}
        \begin{tikzpicture}[auto, scale=1.2]
            \node[circle, fill, inner sep=0pt, minimum size = 1mm] (440460523) at (0.0,0) {};
            \node[circle, fill, inner sep=0pt, minimum size = 1mm] (440460524) at (0.333,0) {};
            \node[circle, fill, inner sep=0pt, minimum size = 1mm] (440460525) at (0.667,0) {};
            \node[circle, fill, inner sep=0pt, minimum size = 1mm] (440460526) at (1.0,0) {};
            \node[circle, fill, inner sep=0pt, minimum size = 1mm] (440460527) at (1.333,0) {};
            \node[circle, fill, inner sep=0pt, minimum size = 1mm] (440460528) at (1.667,0) {};
            \node[circle, fill, inner sep=0pt, minimum size = 1mm] (440460529) at (2.0,0) {};
            \node[circle, fill, inner sep=0pt, minimum size = 1mm] (440460530) at (2.333,0) {};
            \node[circle, fill, inner sep=0pt, minimum size = 1mm] (440460531) at (2.667,0) {};
            \draw (440460523) to[bend right=65, swap] node {$\scriptstyle 1$} (440460531);\draw[double,thick] (440460525) to[bend right=65, swap] node {$\scriptstyle 3$} (440460527);\draw (440460526) to[bend right=75,swap] node {$\scriptstyle 2$} (440460529);\draw (440460528) to[bend right=65,swap] node {$\scriptstyle 4$} (440460530);
        \end{tikzpicture}
    \end{center}
    We have drawn the arcs below the baseline to avoid confusion with the matchings in a clan. These diagrams help explain why Theorem~\ref{thm:cjw} is true. When following a maximal chain up from $\gamma_{p,q}$ to $\gamma$, each matching $(k,n{-}k{+}1)$ in $\gamma_{p,q}$ eventually becomes either a matching in $\gamma$ or a pair of opposite-sign fixed points, which we record as an arc labelled $k$ in the labelled shape.
\end{ex}

It is not hard to give a more direct characterization of the labelled shapes of a clan.
\begin{prop} \label{prop:labelled-shape-rules}
    Let $\omega$ be a partial matching on $[n]$ with its $e$ arcs labelled $1, 2, \ldots, e$, where arcs may be marked or unmarked. Then $\omega$ is a labelled shape for $\gamma \in \Clan_{p,q}$ if and only if $e = \min(p,q)$, and for all arcs $\{i < j\}$ of $\omega$,
    \begin{enumerate}[(i)]
        \item $i$ and $j$ are either matched by $\gamma$ or are a pair of fixed points of opposite sign, according to whether the arc $\{i < j\}$ is unmarked or marked respectively.
        \item If $\{i < j\}$ is marked and $i < i' < j$, then $\omega(i')$ is defined and $\omega(i') < \omega(i) = \omega(j)$.
        \item If $\{i' < j'\}$ is an unmarked arc of $\omega$ with $i' < i < j < j'$, then $\omega(i') = \omega(j') < \omega(i) = \omega(j)$.
    \end{enumerate}
\end{prop}

\begin{proof}
    The number of fixed points remaining in step (b) of Algorithm~\ref{alg:atoms} after all possible pairs $\{i,j\}$ have been removed is
    \begin{equation*}
        |(\text{\# of $+$'s in $\gamma$}) - (\text{\# of $-$'s in $\gamma$})| = |p-q|.
    \end{equation*}
    The number of pairs which were removed is therefore $m = \min(p,q)$, so the image of $\omega$ is $[m]$ and every $k \in [m]$ has $|\omega^{-1}(k)| = 2$.

    If the algorithm removes a pair $i,j$ then it must have already removed all pairs $i',j'$ matched by $\gamma$ with $i' < i < j < j'$, so (iii) is necessary, and if $i,j$ were matched by $\gamma$ then this is the only condition needed for $i,j$ to be removable. To remove a pair $i,j$ fixed by $\gamma$ (so $\{i,j\}$ is marked), one also needs that every $i'$ with $i < i' < j$ has already been removed, meaning $\omega(i') < \omega(i) = \omega(j)$ as demanded by (ii).
\end{proof}

Given an atom $w \in \A(\gamma)$, let $\lsh(w)$ be the corresponding labelled shape. Explicitly, the arcs of $\lsh(w)$ are $\{w^{-1}(k), w^{-1}(n{-}k{+}1)\}$ for $k = 1, 2, \ldots, \min(p,q)$, each arc being marked or unmarked according to whether $w^{-1}(k) > w^{-1}(n{-}k{+}1)$ or $w^{-1}(k) < w^{-1}(n{-}k{+}1)$. Recall that we are trying to characterize the equivalence relation on permutations where $v \sim w$ if $\Gamma(v) = \Gamma(w)$. The set $\Gamma(w)$ is easy to compute from $\lsh(w)$, and in fact the edge labelling on $\lsh(w)$ is not even necessary for this.

\begin{defn} The \emph{unlabelled shape} $\ush(w)$ of $w$ is the pair $(\pi, F)$ where $\pi$ is the partial matching obtained from $\lsh(w)$ by removing the arc labels, and $F$ is the set of endpoints of marked arcs in $\lsh(w)$.
\end{defn}
As before, we consider $\ush(w)$ to be a partial matching with some arcs marked and omit mention of $F$.

\begin{ex} Drawing marked arcs as doubled edges, the unlabelled shape of $w = 157236849$ (whose labelled shape is shown in Example~\ref{ex:labelled-shape}) is
    
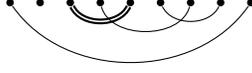
\vspace{0.2cm}
    \begin{center}
        \begin{tikzpicture}[scale=1.2]
            \node[circle, fill, inner sep=0pt, minimum size = 1mm] (4056587) at (0.0,0) {};
            \node[circle, fill, inner sep=0pt, minimum size = 1mm] (4056588) at (0.333,0) {};
            \node[circle, fill, inner sep=0pt, minimum size = 1mm] (4056589) at (0.667,0) {};
            \node[circle, fill, inner sep=0pt, minimum size = 1mm] (4056590) at (1.0,0) {};
            \node[circle, fill, inner sep=0pt, minimum size = 1mm] (4056591) at (1.333,0) {};
            \node[circle, fill, inner sep=0pt, minimum size = 1mm] (4056592) at (1.667,0) {};
            \node[circle, fill, inner sep=0pt, minimum size = 1mm] (4056593) at (2.0,0) {};
            \node[circle, fill, inner sep=0pt, minimum size = 1mm] (4056594) at (2.333,0) {};
            \node[circle, fill, inner sep=0pt, minimum size = 1mm] (4056595) at (2.667,0) {};
            \draw (4056587) to[bend right=55] (4056595);\draw[double, thick] (4056589) to[bend right=65] (4056591);\draw (4056590) to[bend right=75] (4056593);\draw (4056592) to[bend right=65] (4056594);
        \end{tikzpicture}
    \end{center}
\end{ex}

\begin{thm} \label{thm:equivalence-relation-1}
    Let $v$ and $w$ be atoms for some members of $\Clan_{p,q}$. Then $v \sim w$ if and only if $\ush(v) = \ush(w)$.
\end{thm}

\begin{proof} 
    First, $\Gamma(v)$ depends only on $\ush(v)$. Indeed, if $\ush(v)$ has $e$ marked arcs then $\Gamma(v)$ consists of the $2^e$ clans obtained by:
    \begin{itemize}
        \item Replacing each marked arc $\{i,j\}$ by fixed points $i^+, j^-$ or $i^-, j^+$;
        \item Leaving each unmarked arc as a matching;
        \item Leaving each unmatched point as a fixed point whose sign is the sign of $p-q$.
    \end{itemize}
    This shows that if $\ush(v) = \ush(w)$ then $v \sim w$.

    Conversely, suppose $\ush(v) \neq \ush(w)$. If the unmarked arcs of $\ush(v)$ are different from those of $\ush(w)$, then by the previous paragraph every clan in $\Gamma(v)$ has different arcs than every clan in $\Gamma(w)$, so assume all unmarked arcs are the same. Then there must be, say, a marked arc $\{i,j\}$ in $\ush(v)$ such that $i, j$ are not connected by a marked arc in $\ush(w)$. But then there are clans in $\Gamma(w)$ which give the same sign to $i$ and $j$, while every clan in $\Gamma(v)$ gives them opposite signs. In any case, $\Gamma(v) \neq \Gamma(w)$ so $v \not\sim w$.
\end{proof}

An adjacent transposition $s_k$ where $k < \min(p,q)$ acts on a labelled shape $\omega$ by swapping the labels $k$ and $k{+}1$, giving a new partial matching $s_k \omega$ with labelled and possibly marked arcs, although $s_k \omega$ may not be a valid labelled shape for a clan.
\begin{lem} \label{lem:swap-labels} Let $w \in \A(\gamma)$ and $k < \min(p,q)$. Then $s_k \lsh(w)$ is a labelled shape for $\gamma$ if and only if $\ell(s_k s_{n-k} w) = \ell(w)$, and if that holds then $s_k \lsh(w) = \lsh(s_k s_{n-k} w)$. \end{lem}

    \begin{proof} The map $\lsh^{-1}$ sending a labelled shape to its associated atom makes sense when applied to any partial matching with labelled and marked arcs, though the result may not be an atom. In particular, it sends $s_k \lsh(w)$ to the permutation $s_k s_{n-k} w$ regardless of whether the former is a valid labelled shape; here and below, it is helpful to note here that $s_k s_{n-k} = s_{n-k} s_k$ since $k < \min(p,q)$.

        There are two cases in which $s_k \lsh(w)$ is not a valid labelled shape:
        \begin{itemize}
            \item Suppose the arc $\{i < j\}$ of $\lsh(w)$ labeled $k{+}1$ is nested inside the arc $\{i' < j'\}$ labeled $k$, meaning that $i' < i < j < j'$, and that the arc labeled $k$ is unmarked. Then $w$ has the form
            \begin{equation*}
                \cdots k \cdots k{+}1 \cdots n{-}k \cdots n{-}k{+}1 \cdots  \quad \text{or} \quad \cdots k \cdots n{-}k \cdots k{+}1 \cdots n{-}k{+}1 \cdots
            \end{equation*}
            and $\ell(s_k s_{n-k} w) = \ell(w) + 2$.
            \item Suppose the arc $\{i<j\}$ of $\lsh(w)$ labeled $k{+}1$ is marked, and that the arc $\{i'<j'\}$ labeled $k$ has $i < i' < j$ or $i < j' < j$. If $\{i' < j'\}$ is marked, the definition of labeled shape forces $i < i' < j' < j$, so $w$ has the form
            \begin{equation*}
                \cdots n{-}k \cdots n{-}k{+}1 \cdots k \cdots k{+}1\cdots 
            \end{equation*}
            If $\{i'<j'\}$ is unmarked, then depending on exactly where $i',j'$ are positioned with respect to $i,j$, the permutation $w$ has one of the forms 
            \begin{align*}
                &\cdots n{-}k \cdots k \cdots n{-}k{+}1 \cdots  k{+}1 \cdots \\
                &\cdots n{-}k \cdots k \cdots k{+}1 \cdots n{-}k+1 \cdots \\
                &\cdots k \cdots n{-}k \cdots n{-}k+1 \cdots k{+}1 \cdots
            \end{align*}
            In all of these cases, $\ell(s_k s_{n-k} w) = \ell(w)+2$ again.
        \end{itemize}
        Conversely, suppose $\ell(s_k s_{n-k} w) \neq \ell(w)$, so $\ell(s_k s_{n-k} w) = \ell(w) \pm 2$. If $\ell(s_k s_{n-k} w) = \ell(w) + 2$, then $k$ precedes $k+1$ in the one-line notation of $w$ and $n-k$ precedes $n-k+1$. There are $6$ permutations of $k, k{+}1, n{-}k, n{-}k{+}1$ for which this holds, and they are exactly the $6$ cases we considered above in which $s_k \lsh(w)$ is not a valid labelled shape.
        
        So, suppose $\ell(s_k s_{n-k} w) = \ell(w)-2$. We claim that in this case, $\lsh(w)$ could not have been a valid labelled shape to begin with. Now $k+1$ precedes $k$ in $w$ and $n-k+1$ precedes $n-k$, and the $6$ possibilities can be checked directly. If $w$ has one of the forms
        \begin{align*}
            &\cdots n{-}k{+}1 \cdots k{+}1 \cdots k \cdots n{-}k \cdots\\
            &\cdots n{-}k{+}1 \cdots k{+}1 \cdots n{-}k \cdots k \cdots\\
            &\cdots n{-}k{+}1 \cdots n{-}k \cdots k{+}1 \cdots k \cdots\\
            &\cdots k{+}1 \cdots n{-}k{+}1 \cdots k \cdots n{-}k \cdots
        \end{align*}
        then the arc $\{i < j\}$ in $\lsh(w)$ labelled $k$ is marked, yet there is $i < i' < j$ such that $i'$ is labelled $k{+}1$, contradicting Proposition~\ref{prop:labelled-shape-rules}(ii). If $w$ has the form
        \begin{equation*}
            \cdots k{+}1 \cdots k \cdots n{-}k{+}1 \cdots n{-}k \cdots  \quad \text{or} \quad  \cdots k{+}1 \cdots n{-}k{+}1 \cdots k \cdots n{-}k \cdots
        \end{equation*}
        then the arc in $\lsh(w)$ labelled $k$ is nested inside the unmarked arc labelled $k{+}1$, contradicting Proposition~\ref{prop:labelled-shape-rules}(iii).
        \end{proof}

    \begin{thm} \label{thm:equivalence-relation-atoms} The equivalence relation $\sim$ on atoms for $\Clan_{p,q}$ is the transitive closure of the relations $u \sim s_k s_{n-k}u$ where $\ell(s_k s_{n-k}u) = \ell(u)$ and $k < \min(p,q)$. \end{thm}
    \begin{proof} Lemma~\ref{lem:swap-labels} implies that if $\ell(s_k s_{n-k}u) = \ell(u)$ where $k < \min(p,q)$, then $u$ and $s_k s_{n-k}u$ have the same unlabelled shape, so $u \sim s_k s_{n-k} u$ by Theorem~\ref{thm:equivalence-relation-1}.
        
        Conversely, suppose $v \sim w$, so $\ush(v) = \ush(w)$. The labelled shapes $\lsh(v)$ and $\lsh(w)$ are certainly connected by a series of applications of adjacent transpositions, so $v$ and $w$ are connected by transformations $u \mapsto s_k s_{n-k} u$ by Lemma~\ref{lem:swap-labels}, but we must see that this can be done in such a way that all of the intermediate steps are valid labelled shapes.
        
        Proposition~\ref{prop:labelled-shape-rules} shows that the valid labellings of the unlabelled shape $\ush(w)$ can be thought of as the linear extensions of a poset. The elements of the poset are the arcs of $\ush(w)$, and $\{i' < j'\} \leq \{i < j\}$ if either:
        \begin{itemize}
            \item $\{i' < j'\}$ is unmarked and $i' < i < j < j'$; or,
            \item $\{i < j\}$ is marked and $i < i' < j$ or $i < j' < j$.
        \end{itemize}
        Now apply the following general fact: if $P$ is a finite poset and $G$ is the graph whose vertices are the linear extensions $f : P \to [\#P]$ with an edge $(f,g)$ whenever $g = s_i \circ f$, then $G$ is connected.
    \end{proof}

    We can now prove Theorem~\ref{thm:clan-matsumoto-tits}, which we restate here.
    \begin{thm*}[Theorem \ref{thm:clan-matsumoto-tits}] Let $\equiv$ be the equivalence relation on the set of words on $[n-1]$ defined as the transitive closure of the relations $a_1 a_2 \cdots a_{\ell} \equiv (n-a_1)a_2 \cdots a_{\ell}$ together with the Coxeter relations. Every equivalence class of $\equiv$ either contains no reduced word for any $\gamma \in \Clan_{p,q}$, or consists entirely of reduced words. Moreover, $\equiv$ agrees with $\sim$ when restricted to reduced words for members of $\Clan_{p,q}$.
    \end{thm*} 

    \begin{proof}
        Suppose $a = a_1 a_2 \cdots a_\ell \in \Red(\gamma)$ and $b \equiv a$. We know that Coxeter relations preserve $\Red(\gamma)$, so we can assume $b = (n{-}a_1)a_2 \cdots a_{\ell}$. We claim $b \in \Red(\gamma)$ as well. Since $a_1 < \min(p,q)$, it holds that $\gamma_{p,q} \ast s_{a_1}$ is well-defined and equal to $(\cdots (\gamma \ast s_{a_\ell})  \ast \cdots) \ast s_{a_2}$. But also $\gamma_{p,q} \ast s_{a_1} = \gamma_{p,q} \ast s_{n-a_1}$, so
        \begin{align*}
            ((\cdots (\gamma \ast s_{a_\ell})  \ast \cdots) \ast s_{a_2}) \ast s_{n-a_1} &= (\gamma_{p,q} \ast s_{a_1}) \ast s_{n-a_1}\\
            &= (\gamma_{p,q} \ast s_{n-a_1}) \ast s_{n-a_1}\\
            &= \gamma_{p,q},
        \end{align*}
        which implies that $b \in \Red(\gamma)$. We have $a \in \Red(u)$ and $b \in \Red(s_{n-a_1}s_{a_1}u)$ for some $u \in \A(\gamma)$, and $\ell(s_{n-a_1}s_{a_1}u) = \ell(u)$ since $a$ and $b$ have the same length. Theorem~\ref{thm:equivalence-relation-atoms} shows $u \sim s_{n-a_1}s_{a_1}u$, so $a \sim b$. We also conclude from this if an equivalence class of $\equiv$ contains a single reduced word, then all its elements are reduced words.

        Conversely, suppose $a \sim b$ where $a,b$ are reduced words. Then there are atoms $v \sim w$ with $a \in \Red(v)$ and $b \in \Red(w)$, and applying Theorem~\ref{thm:equivalence-relation-atoms}, we can assume that $w = s_{n-k} s_k v$ where $k < \min(p,q)$. Since $s_k s_{n-k} = s_{n-k} s_k$ and $\ell(v) = \ell(s_{n-k} s_k v)$, exactly one of $s_k$ and $s_{n-k}$ is a left descent of $v$, say $s_k$. Then there is $a' \in \Red(v)$ with $a'_1 = k$, and the equality $\ell(v) = \ell(s_{n-k} s_k v)$ implies that $b' := (n{-}a'_1)a'_2\cdots a'_\ell$ is a reduced word of $s_{n-k} s_k v = w$. Since $a$ is related to $a'$ and $b$ to $b'$ via Coxeter relations (by the Matsumoto-Tits lemma), we see that $a \equiv a'\equiv b' \equiv b$.
    \end{proof} 

    Theorem~\ref{thm:clan-matsumoto-tits} can be interpreted as giving a simple prescription for generating each equivalence class making up $\Red(\gamma)$ beginning with one reduced word. It is also natural to ask for simple transformations relating the equivalence classes to each other. We will think about transformations of unlabelled shapes, since these index the equivalence classes of $\sim$ by Theorem~\ref{thm:equivalence-relation-1}. Let $\ush(\A(\gamma))$ be the set of unlabelled shapes for $\gamma$.

    \begin{lem} \label{lem:unlabelled-shape-moves}
        Given an unlabelled shape for $\sigma \in \ush(\A(\gamma))$, suppose one applies to it a transformation of the form
        \begin{center} \raisebox{5.5mm}{$\sigma = $}
            \begin{tikzpicture}[scale=1.2]
                \node at (0,0.15) {$\scriptstyle \alpha$};
                \node at (0.667,0.15) {$\scriptstyle \beta$};
                \node at (1.333,0.15) {$\scriptstyle \alpha$};
                \node at (2,0.15) {$\scriptstyle \beta$};
                \node (4056588) at (-0.36,0) {$\cdots$};
                \node[circle, fill, inner sep=0pt, minimum size = 1mm] (4056587) at (0.0,0) {};
                \node (4056588) at (0.36,0) {$:::$};
                \node[circle, fill, inner sep=0pt, minimum size = 1mm] (4056589) at (0.667,0) {};
                \node (4056590) at (1.0,0) {$:::$};
                \node[circle, fill, inner sep=0pt, minimum size = 1mm] (4056591) at (1.333,0) {};
                \node (4056592) at (1.7,0) {$:::$};
                \node[circle, fill, inner sep=0pt, minimum size = 1mm] (4056593) at (2.0,0) {};
                \node (4056588) at (2.36,0) {$\cdots$};
                \draw[double, thick] (4056587) to[bend right=45] (4056593); \draw[double, thick] (4056589) to[bend right=65] (4056591);
            \end{tikzpicture} \qquad \raisebox{5.5mm}{$\longrightarrow$} \qquad \raisebox{5.5mm}{$\sigma' = $}
                \raisebox{3.5mm}{\begin{tikzpicture}[scale=1.2]
                    \node at (0,0.15) {$\scriptstyle \alpha$};
                    \node at (0.667,0.15) {$\scriptstyle \beta$};
                    \node at (1.333,0.15) {$\scriptstyle \alpha$};
                    \node at (2,0.15) {$\scriptstyle \beta$};
                    \node (4056588) at (-.36,0) {$\cdots$};
                    \node[circle, fill, inner sep=0pt, minimum size = 1mm] (4056587) at (0.0,0) {};
                    \node (4056588) at (0.36,0) {$:::$};
                    \node[circle, fill, inner sep=0pt, minimum size = 1mm] (4056589) at (0.667,0) {};
                    \node (4056590) at (1.0,0) {$:::$};
                    \node[circle,  fill, inner sep=0pt, minimum size = 1mm] (4056591) at (1.333,0) {};
                    \node (4056592) at (1.7,0) {$:::$};
                    \node[circle, fill, inner sep=0pt, minimum size = 1mm] (4056593) at (2.0,0) {};
                    \node (4056588) at (2.36,0) {$\cdots$};
                    \draw[double, thick] (4056587) to[bend right=65] (4056589); \draw[double, thick] (4056591) to[bend right=65] (4056593);
                \end{tikzpicture}}
        \end{center}

        \begin{center} \raisebox{2mm}{$\sigma = $}
            \begin{tikzpicture}[scale=1.2]
                \node at (0,0.15) {$\scriptstyle \alpha$};
                \node at (0.667,0.15) {$\scriptstyle \beta$};
                \node at (1.333,0.15) {$\scriptstyle \alpha$};
                \node (4056588) at (-.36,0) {$\cdots$};
                \node[circle, fill, inner sep=0pt, minimum size = 1mm] (4056587) at (0.0,0) {};
                \node (4056588) at (0.36,0) {$:::$};
                \node[circle, fill, inner sep=0pt, minimum size = 1mm] (4056589) at (0.667,0) {};
                \node (4056590) at (1.0,0) {$:::$};
                \node[circle, fill, inner sep=0pt, minimum size = 1mm] (4056591) at (1.333,0) {};
                \node (4056588) at (1.7,0) {$\cdots$};
                \draw[double, thick] (4056589) to[bend right=65] (4056591);
            \end{tikzpicture} \qquad \raisebox{2mm}{$\longrightarrow$} \qquad \raisebox{2mm}{$\sigma' = $}
            \begin{tikzpicture}[scale=1.2]
                \node at (0,0.15) {$\scriptstyle \alpha$};
                \node at (0.667,0.15) {$\scriptstyle \beta$};
                \node at (1.333,0.15) {$\scriptstyle \alpha$};
                \node (4056588) at (-.36,0) {$\cdots$};
                \node[circle, fill, inner sep=0pt, minimum size = 1mm] (4056587) at (0.0,0) {};
                \node (4056588) at (0.36,0) {$:::$};
                \node[circle, fill, inner sep=0pt, minimum size = 1mm] (4056589) at (0.667,0) {};
                \node (4056590) at (1.0,0) {$:::$};
                \node[circle, fill, inner sep=0pt, minimum size = 1mm] (4056591) at (1.333,0) {};
                \node (4056588) at (1.7,0) {$\cdots$};
                \draw[double, thick] (4056587) to[bend right=65] (4056589);
            \end{tikzpicture} \label{eq:trans2}
            
        \end{center}
        where $\cdots$ conceals an arbitrary partial matching (with marked/unmarked arcs), $:::$ conceals only a \emph{complete} matching (no unpaired fixed points allowed), and $\{\alpha, \beta\} = \{+,-\}$. Then $\sigma' \in \ush(\A(\gamma))$, and the directed graph with vertices $\ush(\A(\gamma))$ and edges $\sigma \to \sigma'$ is acyclic.
    \end{lem}

    \begin{proof}

    Proposition~\ref{prop:labelled-shape-rules} implies that the unlabelled shapes of $\gamma$ are those partial matchings of $[n]$ with $\min(p,q)$ arcs which are either pairs of opposite-sign fixed points or matchings in $\gamma$, and such that no two marked arcs cross and no unpaired fixed point is underneath a marked arc. From this description it is clear that the transformations in the theorem do preserve $\ush(\A(\gamma))$.
    
        Given $\sigma \in \ush(\A(\gamma))$, label the \emph{right} endpoints of the marked arcs $1, 2, \ldots$ from left to right, and then label each marked arc according to its right endpoint. This is a labelled shape of $\gamma$; write $\st(\sigma) \in \A(\gamma)$ for the associated atom. For instance,
        \begin{center} \raisebox{6mm}{$\sigma = $}
            \raisebox{2mm}{\begin{tikzpicture}[scale=1.2]
                \node[circle, fill, inner sep=0pt, minimum size = 1mm] (4056587) at (0.0,0) {};
                \node[circle, fill, inner sep=0pt, minimum size = 1mm] (4056588) at (0.333,0) {};
                \node[circle, fill, inner sep=0pt, minimum size = 1mm] (4056589) at (0.667,0) {};
                \node[circle, fill, inner sep=0pt, minimum size = 1mm] (4056590) at (1.0,0) {};
                \node[circle, fill, inner sep=0pt, minimum size = 1mm] (4056591) at (1.333,0) {};
                \node[circle, fill, inner sep=0pt, minimum size = 1mm] (4056592) at (1.666,0) {};
                \draw[double, thick] (4056587) to[bend right=65] (4056592); \draw[double, thick] (4056588) to[bend right=45] (4056589);  \draw[double, thick] (4056590) to[bend right=45] (4056591);
            \end{tikzpicture}} \quad \raisebox{6mm}{$\leadsto$} \quad
            \begin{tikzpicture}[scale=1.2, auto]
                \node[circle, fill, inner sep=0pt, minimum size = 1mm] (4056587) at (0.0,0) {};
                \node[circle, fill, inner sep=0pt, minimum size = 1mm] (4056588) at (0.333,0) {};
                \node[circle, fill, inner sep=0pt, minimum size = 1mm] (4056589) at (0.667,0) {};
                \node[circle, fill, inner sep=0pt, minimum size = 1mm] (4056590) at (1.0,0) {};
                \node[circle, fill, inner sep=0pt, minimum size = 1mm] (4056591) at (1.333,0) {};
                \node[circle, fill, inner sep=0pt, minimum size = 1mm] (4056592) at (1.666,0) {};
                \draw[double, thick] (4056587) to[bend right=65, swap]  node {$\scriptstyle 3$} (4056592); \draw[double, thick] (4056588) to[bend right=45,swap] node {$\scriptstyle 1$} (4056589);  \draw[double, thick] (4056590) to[bend right=45,swap]  node {$\scriptstyle 2$} (4056591);
            \end{tikzpicture} \qquad \raisebox{6mm}{so $\st(\sigma) = 461523$.}
        \end{center}
        If $\sigma \to \sigma'$, then $\st(\sigma)$ and $\st(\sigma')$ are related by transformations of the form
        \begin{gather*}
            \cdots n{-}\ell{+}1 \cdots n{-}k{+}1 \cdots k \cdots \ell \cdots \to \cdots n{-}k'{+}1 \cdots k' \cdots n{-}\ell{+}1 \cdots \ell \cdots \quad (k' \leq k < \ell)\\
            \cdots j\cdots n{-}k{+}1 \cdots k \cdots \to \cdots n{-}k'{+}1 \cdots k' \cdots j \cdots \quad (k' \leq \min(p,q) < j < \max(p,q))
        \end{gather*}
        In both cases, $\st(\sigma')$ is lexicographically larger than $\st(\sigma)$, which shows that $\ush(\A(\gamma))$ is acyclic.

    \end{proof}
    Lemma~\ref{lem:unlabelled-shape-moves} gives $\ush(\A(\gamma))$ a poset structure, with a covering relation $\sigma \lessdot \sigma'$ when $\sigma \to \sigma'$.
    \begin{ex}
        Here are the posets $\ush(\A(\gamma))$ for $\gamma = --++-+$ and $\gamma = -+-+-+-$:
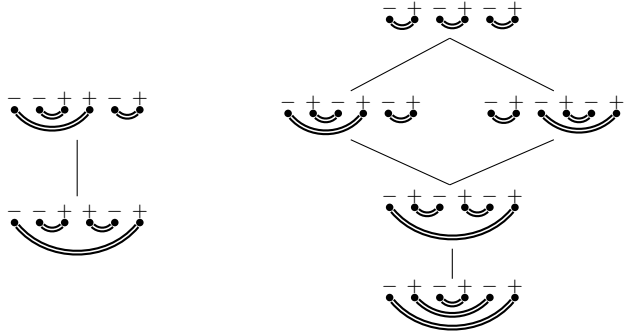
        \begin{center}
            \raisebox{10mm}{\begin{tikzpicture}
                \node at (0,0.15) {$\scriptstyle -$};
                \node at (0.333,0.15) {$\scriptstyle -$};
                \node at (0.666,0.15) {$\scriptstyle +$};
                \node at (1,0.15) {$\scriptstyle +$};
                \node at (1.333,0.15) {$\scriptstyle -$};
                \node at (1.666,0.15) {$\scriptstyle +$};
                \node[circle, fill, inner sep=0pt, minimum size = 1mm] (4056587) at (0.0,0) {};
                \node[circle, fill, inner sep=0pt, minimum size = 1mm] (4056588) at (0.333,0) {};
                \node[circle, fill, inner sep=0pt, minimum size = 1mm] (4056589) at (0.667,0) {};
                \node[circle, fill, inner sep=0pt, minimum size = 1mm] (4056590) at (1.0,0) {};
                \node[circle, fill, inner sep=0pt, minimum size = 1mm] (4056591) at (1.333,0) {};
                \node[circle, fill, inner sep=0pt, minimum size = 1mm] (4056592) at (1.666,0) {};
                \draw[double, thick] (4056587) to[bend right=50] (4056590); \draw[double, thick] (4056588) to[bend right=45] (4056589);  \draw[double, thick] (4056591) to[bend right=45] (4056592);

                \draw (0.833,-1.15) to (0.833,-.4);

                \node at (0,-1.5+0.15) {$\scriptstyle -$};
                \node at (0.333,-1.5+0.15) {$\scriptstyle -$};
                \node at (0.666,-1.5+0.15) {$\scriptstyle +$};
                \node at (1,-1.5+0.15) {$\scriptstyle +$};
                \node at (1.333,-1.5+0.15) {$\scriptstyle -$};
                \node at (1.666,-1.5+0.15) {$\scriptstyle +$};
                \node[circle, fill, inner sep=0pt, minimum size = 1mm] (056587) at (0.0,-1.5) {};
                \node[circle, fill, inner sep=0pt, minimum size = 1mm] (056588) at (0.333,-1.5) {};
                \node[circle, fill, inner sep=0pt, minimum size = 1mm] (056589) at (0.667,-1.5) {};
                \node[circle, fill, inner sep=0pt, minimum size = 1mm] (056590) at (1.0,-1.5) {};
                \node[circle, fill, inner sep=0pt, minimum size = 1mm] (056591) at (1.333,-1.5) {};
                \node[circle, fill, inner sep=0pt, minimum size = 1mm] (056592) at (1.666,-1.5) {};
                \draw[double, thick] (056587) to[bend right=50] (056592); \draw[double, thick] (056588) to[bend right=45] (056589);  \draw[double, thick] (056590) to[bend right=45] (056591);
            \end{tikzpicture}} \qquad \qquad
            \begin{tikzpicture}
                \node at (0,0.15) {$\scriptstyle -$};
                \node at (0.333,0.15) {$\scriptstyle +$};
                \node at (0.666,0.15) {$\scriptstyle -$};
                \node at (1,0.15) {$\scriptstyle +$};
                \node at (1.333,0.15) {$\scriptstyle -$};
                \node at (1.666,0.15) {$\scriptstyle +$};
                \node[circle, fill, inner sep=0pt, minimum size = 1mm] (4056587) at (0.0,0) {};
                \node[circle, fill, inner sep=0pt, minimum size = 1mm] (4056588) at (0.333,0) {};
                \node[circle, fill, inner sep=0pt, minimum size = 1mm] (4056589) at (0.667,0) {};
                \node[circle, fill, inner sep=0pt, minimum size = 1mm] (4056590) at (1.0,0) {};
                \node[circle, fill, inner sep=0pt, minimum size = 1mm] (4056591) at (1.333,0) {};
                \node[circle, fill, inner sep=0pt, minimum size = 1mm] (4056592) at (1.666,0) {};
                \draw[double, thick] (4056587) to[bend right=50] (4056588); \draw[double, thick] (4056589) to[bend right=45] (4056590);  \draw[double, thick] (4056591) to[bend right=45] (4056592);

                \node at (-1.35+0,-1.25+0.15) {$\scriptstyle -$};
                \node at (-1.35+0.333,-1.25+0.15) {$\scriptstyle +$};
                \node at (-1.35+0.666,-1.25+0.15) {$\scriptstyle -$};
                \node at (-1.35+1,-1.25+0.15) {$\scriptstyle +$};
                \node at (-1.35+1.333,-1.25+0.15) {$\scriptstyle -$};
                \node at (-1.35+1.666,-1.25+0.15) {$\scriptstyle +$};
                \node[circle, fill, inner sep=0pt, minimum size = 1mm] (056587) at (-1.35+0.0,-1.25+0) {};
                \node[circle, fill, inner sep=0pt, minimum size = 1mm] (056588) at (-1.35+0.333,-1.25+0) {};
                \node[circle, fill, inner sep=0pt, minimum size = 1mm] (056589) at (-1.35+0.667,-1.25+0) {};
                \node[circle, fill, inner sep=0pt, minimum size = 1mm] (056590) at (-1.35+1.0,-1.25+0) {};
                \node[circle, fill, inner sep=0pt, minimum size = 1mm] (056591) at (-1.35+1.333,-1.25+0) {};
                \node[circle, fill, inner sep=0pt, minimum size = 1mm] (056592) at (-1.35+1.666,-1.25+0) {};
                \draw[double, thick] (056587) to[bend right=50] (056590); \draw[double, thick] (056588) to[bend right=45] (056589);  \draw[double, thick] (056591) to[bend right=45] (056592);

                \node at (1.35+0,-1.25+0.15) {$\scriptstyle -$};
                \node at (1.35+0.333,-1.25+0.15) {$\scriptstyle +$};
                \node at (1.35+0.666,-1.25+0.15) {$\scriptstyle -$};
                \node at (1.35+1,-1.25+0.15) {$\scriptstyle +$};
                \node at (1.35+1.333,-1.25+0.15) {$\scriptstyle -$};
                \node at (1.35+1.666,-1.25+0.15) {$\scriptstyle +$};
                \node[circle, fill, inner sep=0pt, minimum size = 1mm] (6587) at (1.35+0.0,-1.25+0) {};
                \node[circle, fill, inner sep=0pt, minimum size = 1mm] (6588) at (1.35+0.333,-1.25+0) {};
                \node[circle, fill, inner sep=0pt, minimum size = 1mm] (6589) at (1.35+0.667,-1.25+0) {};
                \node[circle, fill, inner sep=0pt, minimum size = 1mm] (6590) at (1.35+1.0,-1.25+0) {};
                \node[circle, fill, inner sep=0pt, minimum size = 1mm] (6591) at (1.35+1.333,-1.25+0) {};
                \node[circle, fill, inner sep=0pt, minimum size = 1mm] (6592) at (1.35+1.666,-1.25+0) {};
                \draw[double, thick] (6587) to[bend right=50] (6588); \draw[double, thick] (6589) to[bend right=45] (6592);  \draw[double, thick] (6590) to[bend right=45] (6591);

                \node at (0,-2.5+0.15) {$\scriptstyle -$};
                \node at (0.333,-2.5+0.15) {$\scriptstyle +$};
                \node at (0.666,-2.5+0.15) {$\scriptstyle -$};
                \node at (1,-2.5+0.15) {$\scriptstyle +$};
                \node at (1.333,-2.5+0.15) {$\scriptstyle -$};
                \node at (1.666,-2.5+0.15) {$\scriptstyle +$};
                \node[circle, fill, inner sep=0pt, minimum size = 1mm] (587) at (0.0,-2.5+0) {};
                \node[circle, fill, inner sep=0pt, minimum size = 1mm] (588) at (0.333,-2.5+0) {};
                \node[circle, fill, inner sep=0pt, minimum size = 1mm] (589) at (0.667,-2.5+0) {};
                \node[circle, fill, inner sep=0pt, minimum size = 1mm] (590) at (1.0,-2.5+0) {};
                \node[circle, fill, inner sep=0pt, minimum size = 1mm] (591) at (1.333,-2.5+0) {};
                \node[circle, fill, inner sep=0pt, minimum size = 1mm] (592) at (1.666,-2.5+0) {};
                \draw[double, thick] (587) to[bend right=50] (592); \draw[double, thick] (588) to[bend right=45] (589);  \draw[double, thick] (590) to[bend right=45] (591);

                \node at (0,-3.7+0.15) {$\scriptstyle -$};
                \node at (0.333,-3.7+0.15) {$\scriptstyle +$};
                \node at (0.666,-3.7+0.15) {$\scriptstyle -$};
                \node at (1,-3.7+0.15) {$\scriptstyle +$};
                \node at (1.333,-3.7+0.15) {$\scriptstyle -$};
                \node at (1.666,-3.7+0.15) {$\scriptstyle +$};
                \node[circle, fill, inner sep=0pt, minimum size = 1mm] (87) at (0.0,-3.7+0) {};
                \node[circle, fill, inner sep=0pt, minimum size = 1mm] (88) at (0.333,-3.7+0) {};
                \node[circle, fill, inner sep=0pt, minimum size = 1mm] (89) at (0.667,-3.7+0) {};
                \node[circle, fill, inner sep=0pt, minimum size = 1mm] (90) at (1.0,-3.7+0) {};
                \node[circle, fill, inner sep=0pt, minimum size = 1mm] (91) at (1.333,-3.7+0) {};
                \node[circle, fill, inner sep=0pt, minimum size = 1mm] (92) at (1.666,-3.7+0) {};
                \draw[double, thick] (87) to[bend right=50] (92); \draw[double, thick] (88) to[bend right=45] (91);  \draw[double, thick] (89) to[bend right=45] (90);

                \draw (-1.35+0.833,-.95) to (0.8,-.25);
                \draw (1.35+0.833,-.95) to (0.8,-.25);
                \draw (-1.35+0.833,-1.6) to (0.8,-2.2);
                \draw (1.35+0.833,-1.6) to (0.8,-2.2);
                \draw (0.833,-3.05) to (0.833,-3.45);
            \end{tikzpicture}
        \end{center}
    \end{ex}

    \begin{thm} \label{thm:unique-max-ush} The poset $\ush(\A(\gamma))$ has a unique maximal element $\sigma_{\max}$, which can be constructed as follows. First, $\sigma_{\max}$ has an unmarked arc for every matching of $\gamma$. Next, find the minimal fixed point $i$ of $\gamma$ such that the minimal fixed point $j > i$ has opposite sign, and connect $i$ and $j$ by a marked arc in $\sigma_{\max}$; repeat this process, ignoring fixed points that have already been connected, until all remaining fixed points have the same sign.
    \end{thm}

    \begin{proof} The unmarked arcs in $\sigma\in \ush(\A(\gamma))$ are determined by the matchings of $\gamma$, and play no role in the poset structure. We may therefore ignore them, and assume that $\gamma$ is matchless. If $\gamma$ has no pair of fixed points of opposite sign, then $\ush(\A(\gamma))$ has one element, so the theorem is trivially true. Otherwise, let $i$ be minimal such that $\gamma(i)$ and $\gamma(i+1)$ have opposite sign, and define $\bar\gamma \in \Clan_{p-1,q-1}$ by removing $i$ and $i+1$ from the matching diagram of $\gamma$.

    There is an injection $f : \ush(\A(\bar \gamma)) \to \ush(\A(\gamma))$ which adds the fixed points $i,i{+}1$ back and connects them with a marked arc. For example, if $\gamma = ++--+$ then $\bar\gamma = +-+$, and
    \begin{center}
        \vspace{-1.5mm}
        \raisebox{3mm}{$f : $}\, \raisebox{2.5mm}{\begin{tikzpicture}[scale=1.2]
            \node at (0,0.15) {$\scriptstyle +$};
            \node at (0.333,0.15) {$\scriptstyle -$};
            \node at (0.666,0.15) {$\scriptstyle +$};
            \node[circle, fill, inner sep=0pt, minimum size = 1mm] (4056587) at (0.0,0) {};
            \node[circle, fill, inner sep=0pt, minimum size = 1mm] (4056590) at (0.333,0) {};
            \node[circle, fill, inner sep=0pt, minimum size = 1mm] (4056591) at (0.666,0) {};
            \draw[double, thick] (4056587) to[bend right=50] (4056590);
        \end{tikzpicture}} \quad \raisebox{3mm}{$\mapsto$} \quad \begin{tikzpicture}[scale=1.2]
            \node at (0,0.15) {$\scriptstyle +$};
            \node at (0.333,0.15) {$\scriptstyle +$};
            \node at (0.666,0.15) {$\scriptstyle -$};
            \node at (1,0.15) {$\scriptstyle -$};
            \node at (1.333,0.15) {$\scriptstyle +$};
            \node[circle, fill, inner sep=0pt, minimum size = 1mm] (4056587) at (0.0,0) {};
            \node[circle, fill, inner sep=0pt, minimum size = 1mm] (4056588) at (0.333,0) {};
            \node[circle, fill, inner sep=0pt, minimum size = 1mm] (4056589) at (0.667,0) {};
            \node[circle, fill, inner sep=0pt, minimum size = 1mm] (4056590) at (1.0,0) {};
            \node[circle, fill, inner sep=0pt, minimum size = 1mm] (4056591) at (1.333,0) {};
            \draw[double, thick] (4056587) to[bend right=50] (4056590); \draw[double, thick] (4056588) to[bend right=45] (4056589); 
        \end{tikzpicture}
        \vspace{-1.5mm}
    \end{center}
    By induction, $\ush(\A(\bar \gamma))$ has a unique maximal element $\sigma_{\max}'$ constructed as described in the theorem. Its image $f(\sigma_{\max}')$ equals the unlabeled shape $\sigma_{\max}$, which we must now see is actually the unique maximal element of $\ush(\A(\gamma))$. First, suppose $\sigma \in \ush(\A(\gamma))$ is in the image of $f$, or equivalently that $\sigma$ has $\{i < i+1\}$ as a marked arc.  Since $f$ does not add any unpaired fixed points, any transformation which can be performed in $\ush(\A(\bar \gamma))$ can also be performed in $\ush(\A(\gamma))$, so $f$ is a poset homomorphism. This implies $\sigma \leq f(\sigma_{\max}') = \sigma_{\max}$.

    Now suppose $\sigma$ is not in the image of $f$; we claim $\sigma$ cannot be maximal. Consider two cases:
    \begin{itemize}
        \item Suppose $\sigma$ pairs $i$ with $j'$ and $i{+}1$ with $j$. Then $i+1 < j < j'$, for otherwise there would be an unpaired fixed point of $\sigma$ below a marked arc, or else two marked arcs would cross. That is, $\sigma$ has the form
        
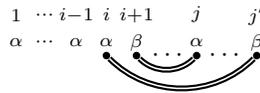
\begin{center}
            \begin{tikzpicture}[scale=1.2]
                \node at (-1,0.45) {$\scriptstyle 1$};
                \node at (-0.666,0.45) {$\scriptstyle \cdots$};
                \node at (-0.333,0.45) {$\scriptstyle i{-}1$};
                \node at (0,0.46) {$\scriptstyle i$};
                \node at (0.333,0.45) {$\scriptstyle i{+}1$};
                \node at (1,0.46) {$\scriptstyle j$};
                \node at (1.666,0.45) {$\scriptstyle j'$};

                \node at (-1,0.15) {$\scriptstyle \alpha$};
                \node at (-0.666,0.15) {$\scriptstyle \cdots$};
                \node at (-0.333,0.15) {$\scriptstyle \alpha$};
                \node at (0,0.15) {$\scriptstyle \alpha$};
                \node at (0.333,0.15) {$\scriptstyle \beta$};
                \node at (1,0.15) {$\scriptstyle \alpha$};
                \node at (1.666,0.15) {$\scriptstyle \beta$};
                \node[circle, fill, inner sep=0pt, minimum size = 1mm] (4056587) at (0.0,0) {};
                \node[circle, fill, inner sep=0pt, minimum size = 1mm] (4056589) at (0.333,0) {};
                \node (4056590) at (0.7,0) {$\cdots$};
                \node[circle, fill, inner sep=0pt, minimum size = 1mm] (4056591) at (1,0) {};
                \node (4056592) at (1.333,0) {$\cdots$};
                \node[circle, fill, inner sep=0pt, minimum size = 1mm] (4056593) at (1.666,0) {};
                \draw[double, thick] (4056587) to[bend right=45] (4056593); \draw[double, thick] (4056589) to[bend right=45] (4056591);
            \end{tikzpicture}
        \end{center}
        where $\{\alpha,\beta\} = \{+,-\}$ and there are no unpaired fixed points in $[i,j']$. But now we can apply the transformation replacing the marked arcs $\{i<j'\}, \{i{+}1<j\}$ by $\{i<i{+}1\}, \{j<j'\}$, so $\sigma$ is not maximal.

        \item Suppose one of $i, i{+}1$ is unpaired in $\sigma$ (they cannot both be unpaired). Then in fact $i$ must be unpaired, because otherwise it would have to be paired with some $j > i{+}1$, but then the unpaired fixed point $i{+}1$ would be below the marked arc $\{i < j\}$. So, say $i{+}1$ is paired with $j$. We must have $j > i{+}1$, because otherwise the unpaired fixed point $i$ would be below the marked arc $\{j < i{+}1\}$. That is, $\sigma$ has the form
        \begin{center}
            \begin{tikzpicture}[scale=1.2]
                \node at (-1,0.45) {$\scriptstyle 1$};
                \node at (-0.666,0.45) {$\scriptstyle \cdots$};
                \node at (-0.333,0.45) {$\scriptstyle i{-}1$};
                \node at (0,0.46) {$\scriptstyle i$};
                \node at (0.333,0.45) {$\scriptstyle i{+}1$};
                \node at (1,0.46) {$\scriptstyle j$};

                \node at (-1,0.15) {$\scriptstyle \alpha$};
                \node at (-0.666,0.15) {$\scriptstyle \cdots$};
                \node at (-0.333,0.15) {$\scriptstyle \alpha$};
                \node at (0,0.15) {$\scriptstyle \alpha$};
                \node at (0.333,0.15) {$\scriptstyle \beta$};
                \node at (1,0.15) {$\scriptstyle \alpha$};
                \node[circle, fill, inner sep=0pt, minimum size = 1mm] (4056587) at (0.0,0) {};
                \node[circle, fill, inner sep=0pt, minimum size = 1mm] (4056589) at (0.333,0) {};
                \node (4056590) at (0.7,0) {$\cdots$};
                \node[circle, fill, inner sep=0pt, minimum size = 1mm] (4056591) at (1,0) {};
                \draw[double, thick] (4056589) to[bend right=45] (4056591);
            \end{tikzpicture}
        \end{center}
        Now we can apply the transformation replacing the marked arc $\{i{+}1 < j\}$ by $\{i < i{+}1\}$, so $\sigma$ is not maximal.
    \end{itemize}
    \end{proof}

    Theorem~\ref{thm:unique-max-ush} gives a prescription for generating all of $\ush(\A(\gamma))$ from one element $\sigma_{\max}$ by applying simple transformations. It would be interesting to be able to do this at the level of reduced words: that is, to give a uniform way of beginning with a relation $\ush(v) \to \ush(w)$ and producing $a \in \Red(v)$ and $b \in \Red(w)$ which are related in some simple way.

\section{Enumerating reduced words for clans}
\label{sec:enum}

\begin{defn} Let $a = a_1 \cdots a_\ell$ be a word on the alphabet $\NN$. A \emph{compatible sequence} for $a$ is a word $b$ of length $\ell$ such that
    \begin{itemize}
        \item $1 \leq b_1 \leq \cdots \leq b_\ell$
        \item $b_i \leq a_i$ for each $i$
        \item For each $i$, if $a_i < a_{i+1}$, then $b_i < b_{i+1}$.
    \end{itemize}
    We use bold for compatible sequences just as for reduced words.
\end{defn}
Let $\comp(a)$ be the set of compatible sequences for $a$. For instance, $\comp(\bf 3213) = \{\bf 1112, 1113\}$ while $\comp(\bf 3231)$ is empty.

\begin{defn}  \label{defn:schubert} The \emph{Schubert polynomial} of a permutation $w \in S_n$ is
    \begin{equation*}
        \S_w = \sum_{a \in \Red(w)} \sum_{b \in \comp(a)} x_{b_1} \cdots x_{b_\ell}.
    \end{equation*}
    The \emph{Stanley symmetric function} of $w$ is the formal power series $F_w = \lim_{m \to \infty} \S_{w^{+m}}$, where $w^{+m}$ is the permutation defined inductively by $w^{+m} = (w^{+(m-1)})^{+1}$ and $w^{+1} = 1(w_1+1) \cdots (w_n+1)$ in one-line notation.
\end{defn}
It is not hard to check that $\lim_{m \to \infty} \S_{w^{+m}}$ does exist as a formal power series, so that $F_w$ is well-defined. The fact that it is actually a symmetric function is rather less obvious, and was proved by Stanley \cite{stanleysymm}. 

\begin{defn} \label{defn:clan-schubert} The \emph{Schubert polynomial} of a $(p,q)$-clan $\gamma$ is
    \begin{equation*}
        \clS_\gamma = \sum_{a \in \clRed(\gamma)} \sum_{b \in \comp(a)} x_{b_1} \cdots x_{b_\ell}.
    \end{equation*}
    The \emph{Stanley symmetric function} of $\gamma$ is $\lim_{m\to \infty} \clS_{\gamma^{+m}}$. Here $\gamma^{+m}$ is the $(p+m,q+m)$-clan defined inductively by $\gamma^{+m} = (\gamma^{+(m-1)})^{+1}$ and where $\gamma^{+1}$ is obtained from $\gamma$ by shifting all of $1, 2, \ldots, n$ up by one and then multiplying by the cycle $(1\,\,n{+}2)$.
\end{defn}
\newpage
\begin{ex} As per Example~\ref{ex:atoms},
    \begin{equation*}
        \clRed(+--+) = \{\bf 3213, 3231, 2321, 1213, 1231, 2123\}.
    \end{equation*}
    $\comp(a)$ is empty for all $a \in \clRed(+--+)$ except $\bf 3213$ and $\bf 2123$, while $\comp(\bf 3213) = \{\bf 1112, 1113\}$ and $\comp(\bf 2123) = \{\bf 1123\}$. Thus $\clS_{+--+} = x_1^3 x_2 + x_1^3 x_3 + x_1^2 x_2 x_3$. Also,
    \begin{equation*}
        (+--+)^{+1} \,=\, \raisebox{-2mm}{\begin{tikzpicture}
            \node[circle, fill, inner sep=0pt, minimum size = 1mm] (790359857) at (0.0,0) {}; 
            \node (790359858) at (0.333,0) {$+$}; 
            \node (790359859) at (0.667,0) {$-$}; 
            \node (790359860) at (1.0,0) {$-$}; 
            \node (790359861) at (1.333,0) {$+$}; 
            \node[circle, fill, inner sep=0pt, minimum size = 1mm] (790359862) at (1.667,0) {}; 
            \draw (790359857) to[bend left=82] (790359862); 
        \end{tikzpicture}}
    \end{equation*}
\end{ex}
Since $\clRed(\gamma) = \bigcup_{w \in \A(\gamma)} \Red(w)$ we have $\clS_\gamma = \sum_{w \in \A(\gamma)} \S_w$.

\begin{prop} \label{prop:stable-limit} $\A(\gamma^{+1}) = \{w^{+1} : w \in \A(\gamma)\}$, and $F_{\gamma}$ is a well-defined symmetric function equal to $\sum_{w \in \A(\gamma)} F_w$. \end{prop}
    \begin{proof} Given a word $a = a_1 \cdots a_{\ell}$, let $a^{+1} = (a_1+1)\cdots (a_\ell+1)$. It is clear that if $a \in \clRed(\gamma)$, then $a^{+1} \in \clRed(\gamma^{+1})$.
        
        Conversely, if $\gamma \ast s_i < \gamma$, one sees from Figure~\ref{fig:local-moves} that the size of the largest cycle in $\gamma$---i.e. the maximum of $|j-i|$ over all 2-cycles $(i\,j)$ in $\gamma$---is no larger than the size of the largest cycle in $\gamma \ast s_i$. This implies that if some $\delta \in \Clan_{p+1,q+1}$ has a reduced word containing the letter $1$ or $n+1$, then the largest cycle in $\delta$ has size $< n+1$. Since $\gamma^{+1}$ does have a cycle of size $n+1$, its reduced words are supported on the alphabet $\{2, 3, \ldots, n\}$, and so they must all have the form $a^{+1}$ for some $a \in \clRed(\gamma)$. This proves $\A(\gamma^{+1}) = \{w^{+1} : w \in \A(\gamma)\}$, and now
        \begin{equation*}
            F_\gamma = \lim_{m \to \infty} \sum_{w \in \A(\gamma^{+m})} \S_w = \lim_{m \to \infty} \sum_{w \in \A(\gamma)} \S_{w^{+m}} = \sum_{w \in \A(\gamma)} F_w.
        \end{equation*}
    \end{proof}

\begin{prop} \label{prop:reduced-word-count} Letting $\ell$ be the degree of $F_{\gamma}$, the coefficient of $x_1 x_2 \cdots x_{\ell}$ in $F_{\gamma}$ is $\#\Red(\gamma)$. \end{prop}
    \begin{proof} If $m \geq \ell - 1$, then every letter of $a^{+m}$ for $a \in \clRed(\gamma)$ is at least $\ell$, and so $\comp(a^{+m})$ contains $\bf 12\cdots \ell$. Proposition~\ref{prop:stable-limit} therefore shows that the coefficient of $x_1 x_2 \cdots x_{\ell}$ in $\clS_{\gamma^{+m}}$ is $\#\Red(\gamma^{+m}) = \#\Red(\gamma)$ for all $m \geq \ell-1$.
    \end{proof} 

This proposition holds equally well for Stanley symmetric functions of permutations, which was Stanley's motivation for defining $F_w$. One can then use symmetric function techniques to extract coefficients of $F_w$ and enumerate $\Red(w)$. For instance, Stanley showed that $F_{n\cdots 21}$ is the Schur function $s_{(n-1,n-2,\ldots,1)}$, and comparing coefficients of $x_1 x_2 \cdots$ shows that $\#\Red(n\cdots 21)$ equals the number of standard tableaux of shape $(n-1,n-2,\ldots,1)$ \cite{stanleysymm}. Our intent is to do the same for clans.

\begin{defn} For $1 \leq i < n$, the \emph{divided difference operator} $\partial_i$ acting on $R[x_1, \ldots, x_n]$ for a commutative ring $R$ sends $f$ to $\partial_i f = (f - s_i f)/(x_i - x_{i+1})$, where $s_i$ acts on $R[x_1, \ldots, x_n]$ by swapping $x_i$ and $x_{i+1}$. The \emph{isobaric divided difference operator} $\pi_i$ is defined by $\pi_i(f) = \partial_i(x_i f)$. \end{defn}

Lascoux and Sch\"utzenberger defined Schubert polynomials for $S_n$ by setting $\S_{n\cdots 21} = x_1^{n-1} x_2^{n-2} \cdots x_{n-1}$ and then using the recurrence $\S_{ws_i} = \partial_i \S_w$ to define all the other polynomials by induction on weak order \cite{lascouxschutzenbergerschubert}. Earlier work of Bernstein, Gelfand, and Gelfand \cite{BGG-schubert} shows that this definition ensures that $\S_w$ represents the cohomology class of the Schubert variety in $\Fl(\CC^n)$ labeled by $w$. It is a theorem of Billey, Jockusch, and Stanley \cite{billeyjockuschstanley} that this definition is equivalent to Definition~\ref{defn:schubert}.

Wyser and Yong \cite{wyser-yong-clans} define clan Schubert polynomials using the same strategy: they give an explicit formula when the clan is matchless, and apply divided difference operators to produce the polynomials for other clans by induction on weak order. Their formulas are given in terms of flagged Schur polynomials, which we now define.

\begin{defn} Let $X_k$ be the alphabet $\{x_1, \ldots, x_k\}$. Let $\lambda$ be a partition and $\phi$ a sequence of natural numbers of the same length. The \emph{flagged Schur polynomial} of shape $\lambda$ with flag $\phi$ is the polynomial $\sum_T x^T$ where $T$ runs over all semistandard tableaux of shape $T$ whose entries in each row $i$ come from $\{1, 2, \ldots, \phi_i\}$, and as usual $x^T$ is the content monomial $\prod_i x_i^{\text{\# of $i$'s in $T$}}$. We write $s_{\lambda}(X_{\phi_1}, \ldots, X_{\phi_\ell})$ or just $s_{\lambda}(X_\phi)$ for this polynomial. \end{defn}

\begin{ex} \hfill
\begin{itemize}
\item $s_{21}(X_1, X_2) = x_1^2 x_2$ is a sum over the single tableau $\scriptsize \young(11,2)$.
\item $s_{21}(X_2, X_2) = x_1^2 x_2 + x_1 x_2^2$ is a sum over the two tableaux $\scriptsize \young(11,2)$ and $\scriptsize \young(12,2)$.
\item $s_{21}(X_1, X_1) = 0$.
\end{itemize}
\end{ex}

\begin{defn} Given a matchless $(p,q)$-clan $\gamma$, let $\phi^+(\gamma)$ be the list of positions of the $+$'s in increasing order, and likewise $\phi^-(\gamma)$ the list of positions of the $-$'s. Also define two partitions $\lambda^+(\gamma)$ and $\lambda^-(\gamma)$ by
    \begin{align*}
        &\lambda^+(\gamma)_i = \#\{j : \phi^-(\gamma)_j > \phi^+(\gamma)_i\} \qquad \text{for $i = 1, \ldots, p$}\\
         &\lambda^-(\gamma)_i = \#\{j : \phi^+(\gamma)_j > \phi^-(\gamma)_i\} \qquad \text{for $i = 1, \ldots, q$}. 
    \end{align*}
\end{defn}
The map $\phi^+(\gamma) \mapsto \lambda^+(\gamma)$ is a bijection between $p$-subsets of $[p+q]$ and partitions whose Young diagram is contained in the $p \times q$ rectangle $[p] \times [q]$. Graphically, if one labels the $p+q$ segments of the southwest boundary of the Young diagram of $\lambda^+(\gamma)$ with $1, 2, \ldots, p+q$ from top to bottom, the set of vertical segments is $\phi^+(\gamma)$ and the set of horizontal segments is $\phi^-(\gamma)$.

\begin{ex} Letting $\gamma = +--+-+++-$, we have
    \begin{center}
    \begin{tabular}{c}
        $\phi^- = (\textcolor{blue}{2},\textcolor{blue}{3},\textcolor{blue}{5},\textcolor{blue}{9})$\\
        $\lambda^- = (4,4,3,0)$
    \end{tabular} \quad and \quad
    \begin{tabular}{ccc}
        $\phi^+ = (\textcolor{red}{1},\textcolor{red}{4},\textcolor{red}{6},\textcolor{red}{7},\textcolor{red}{8})$\\
        $\lambda^+ = (4,2,1,1,1)$
    \end{tabular} \quad and \quad $\lambda^+ = $ \raisebox{-12mm}
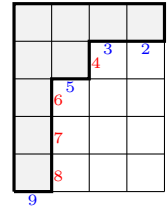
{ 
    \begin{tikzpicture}[scale=0.5]
            \draw[very thick, fill=black!5!white] (0,0) to node[color=blue, yshift=-3pt] {\tiny 9} (1,0) to node[color=red, xshift=2.5pt, yshift=-1pt] {\tiny 8} (1,1) to node[color=red, xshift=2.5pt, yshift=-1pt] {\tiny 7} (1,2) to node[color=red, xshift=2.5pt, yshift=-1pt] {\tiny 6} (1,3) to node[color=blue, yshift=-3pt] {\tiny 5} (2,3) to node[color=red, xshift=2.5pt, yshift=-1pt] {\tiny 4} (2,4) to node[color=blue, yshift=-3pt] {\tiny 3} (3,4) to node[color=blue, yshift=-3pt] {\tiny 2} (4,4) to node[color=red, xshift=2.5pt] {\tiny 1} (4,5) to (0,5) to (0,0);
            \draw (0,0) to (4,0);
            \draw (0,1) to (4,1);
            \draw (0,2) to (4,2);
            \draw (0,3) to (4,3);
            \draw (0,4) to (4,4);
            \draw (0,5) to (4,5);
            \draw (0,0) to (0,5);
            \draw (1,0) to (1,5);
            \draw (2,0) to (2,5);
            \draw (3,0) to (3,5);
            \draw (4,0) to (4,5);
            
        \end{tikzpicture} }
    \end{center}
\end{ex}

If $\lambda \subseteq [p] \times [q]$, write $\lambda^\vee$ for the partition whose Young diagram is the complement of $\lambda$ in $[p] \times [q]$ (rotated $180^\circ$). Also let $\lambda^t$ denote the partition conjugate to $\lambda$. For a matchless $(p,q)$-clan $\gamma$, let $\rev(\gamma)$ be the clan obtained by reversing the one-line notation of $\gamma$. Let $\neg(\gamma)$ be the $(q,p)$-clan obtained by switching the signs of all fixed points in $\gamma$. The next proposition is clear from the description above of the map $\gamma \mapsto \lambda^+(\gamma)$ in terms of lattice paths.
\begin{prop} \label{prop:symmetries} $\lambda^+(\gamma)^\vee = \lambda^+(\rev(\gamma))$ and $\lambda^+(\gamma)^t = \lambda^+(\neg(\rev(\gamma))$, and therefore
\begin{equation*}
\lambda^-(\gamma) = \lambda^+(\neg(\gamma)) = (\lambda^+(\gamma)^t)^\vee.
\end{equation*} \end{prop}

\begin{defn} \label{defn:wyser-yong-schubert} The \emph{Wyser-Yong Schubert polynomials} labeled by the members of $\Clan_{p,q}$ are defined by induction on clan weak order using the recurrence
    \begin{center} \renewcommand{\arraystretch}{1.5}
    \begin{tabular}{ll}
        $\displaystyle \WYS_\gamma = s_{\lambda^+(\gamma)}(X_{\phi^+(\gamma)})\, s_{\lambda^-(\gamma)}(X_{\phi^-(\gamma)})$ & if $\gamma$ is matchless\\
        $\displaystyle \WYS_{\gamma \ast s_i} = \partial_i \WYS_{\gamma}$ & if $\gamma \ast s_i < \gamma$.
    \end{tabular}
\end{center}
 \end{defn}

\begin{thm}[\cite{wyser-yong}] \label{thm:wyser-yong-consistency} Definition~\ref{defn:wyser-yong-schubert} makes sense:  given a fixed $\gamma$, the polynomial $\WYS_\gamma$ is independent of the choice of matchless clan $\gamma'$ and saturated chain
\begin{equation*}
\gamma < \cdots < (\gamma' \ast s_{a_1}) \ast s_{a_2} < \gamma' \ast s_{a_1} < \gamma'.
\end{equation*}
used to compute it. Also, if $\gamma \ast s_i \not< \gamma$, then $\partial_i \WYS_\gamma = 0$ (this includes the case where $\gamma \ast s_i$ is not defined). \end{thm}

Wyser and Yong also show that $\WYS_\gamma$ represents the cohomology class  $[\overline{Y}_{\gamma}]$. Brion \cite{brion} had previously given a formula for $[\overline{Y}_{\gamma}]$ as a sum of Schubert classes, from which one can deduce a formula for $\WYS_\gamma$ as a sum of Schubert polynomials. In fact, this formula is simply Definition~\ref{defn:clan-schubert}, so the next lemma is not really new, but we include a self-contained proof because it is not entirely obvious that the summands in Brion's formula are indeed the $\S_w$ for $w \in \A(\gamma)$.

\begin{rem} The last claim in Theorem~\ref{thm:wyser-yong-consistency}, that $\partial_i \WYS_\gamma = 0$ if $\gamma \ast s_i \not< \gamma$, is not stated explicitly in \cite{wyser-yong-clans}, but it follows from the geometry. Indeed, the geometric interpretation of weak order mentioned in Remark~\ref{rem:geom-weak-order} is essentially that if $\partial_i [\overline{Y}_\gamma]$ is nonzero, then it equals some $[\overline{Y}_{\gamma'}]$ and then one takes $\gamma' < \gamma$ to be a covering in weak order labeled by $s_i$.
\end{rem}

\begin{lem} $\clS_\gamma = \WYS_\gamma$ for any clan $\gamma \in \Clan_{p,q}$. \end{lem}

\begin{proof}    We claim that $\clS_\gamma$ satisfies the same recurrence as $\WYS_{\gamma}$, namely, if $1 \leq i < n$, then 
    \begin{equation} \label{eq:clan-schubert-recurrence}
        \partial_i \clS_\gamma = \begin{cases}
            \S_{\gamma \ast s_i} & \text{if $\gamma \ast s_i < \gamma$}\\
            0 & \text{otherwise}
        \end{cases}.
    \end{equation}
    Given that $\clS_\gamma = \sum_{w \in \A(\gamma)} \S_w$ and that ordinary Schubert polynomials satisfy the recurrence
    \begin{equation*}
        \partial_i \S_w = \begin{cases}
            \S_{ws_i} & \text{if $\ell(ws_i) < \ell(w)$}\\
            0 & \text{otherwise}
        \end{cases},
    \end{equation*}
    this claim follows from two simple facts about atoms:
    \begin{enumerate}[(i)]
        \item If $\gamma \ast s_i < \gamma$, then $\A(\gamma \ast s_i) = \{ws_i : w \in \A(\gamma) \text{ and } \ell(ws_i) < \ell(w)\}$.
        \item If $\gamma \ast s_i \not< \gamma$, then $\ell(ws_i) > \ell(w)$ for all $w \in \A(\gamma)$.
    \end{enumerate}
    If $w \in \A(\gamma)$ and $\ell(ws_i) < \ell(w)$, then $w$ has a reduced word ending in $\bf i$, so $\gamma \ast s_i < \gamma$; this proves (ii). As for (i), if $\gamma \ast s_i < \gamma$ then $\Red(\gamma \ast s_i) = \{a_1 \cdots a_\ell : a_1 \cdots a_\ell {\bf i} \in \Red(\gamma)\}$, which is equivalent to (i).

    Now we show that $\clS_\gamma = \WYS_\gamma$ for all $\gamma$ by induction on the rank of $\gamma$ in weak order. The base case is $\clS_{\gamma_{p,q}} = \WYS_{\gamma_{p,q}} = 1$ (the second equality is clear from the geometry, if not from Definition~\ref{defn:wyser-yong-schubert}). Equation~\eqref{eq:clan-schubert-recurrence} and Theorem~\ref{thm:wyser-yong-consistency} show that for any $i < n$,
    \begin{equation*}
    \partial_i(\clS_\gamma - \WYS_\gamma) = 0,
    \end{equation*}
    using that $\clS_{\gamma \ast s_i} = \WYS_{\gamma \ast s_i}$ by induction. The kernel of $\partial_i$ on $\ZZ[x_1, \ldots, x_n]$ consists of those polynomials symmetric in $x_i$ and $x_{i+1}$, so this shows $\clS_\gamma - \WYS_\gamma$ is symmetric in $x_1, \ldots, x_n$. By \cite[Proposition 2.9]{wyser-yong-clans}, $\WYS_\gamma$ is a linear combination of Schubert polynomials $\S_w$ for $w \in S_n$, so the same is true of $\clS_\gamma - \WYS_\gamma$. But it is well-known that the Schubert polynomials $\S_w$ for $w \in S_n$ are linearly independent modulo the ideal in $\ZZ[x_1, \ldots, x_n]$ generated by symmetric polynomials \cite[\sectionsymbol 2.5.2]{manivel}, so $\clS_\gamma - \WYS_\gamma = 0$. 
\end{proof}

Our next goal is to leverage the formulas of Wyser and Yong to prove enumerative results for clan words via the Stanley symmetric functions $F_{\gamma}$. To do this, we must better understand the procedure of passing from $\clS_\gamma$ to $F_{\gamma}$. An important fact about the divided difference operators $\partial_i$ is that they satisfy the braid relations for $S_n$: that is, $\partial_i \partial_k = \partial_k \partial_i$ if $|i-k| > 1$ and $\partial_i \partial_j \partial_i = \partial_j \partial_i \partial_j$ if $|i-j| = 1$. As a consequence, we can define $\partial_w$ as the composition $\partial_{a_1} \cdots \partial_{a_\ell}$ for a reduced word $a \in \Red(w)$, and the resulting operator is independent of the choice of $a$. The same holds for the $\pi_i$.

\begin{lem}[\cite{HMP1}, Theorem 3.40; \cite{M2}, equation (4.25)] \label{lem:stabilization} Let $w_n = n(n-1)\cdots 21 \in S_n$. If $f \in \ZZ[x_1, \ldots, x_n]$ and $N \geq n$, then $\pi_{w_N} f$ is a symmetric polynomial in $x_1, \ldots, x_N$. Moreover, $\lim_{N \to \infty} \pi_{w_N} \S_w = F_w$ for any permutation $w$. \end{lem}
It follows by linearity that $\pi_{w_{n}} \clS_\gamma = F_\gamma$ for $\gamma \in \Clan_{p,q}$. The result we are working towards is that, if $\gamma$ is matchless, then $F_\gamma = s_{\lambda^+(\gamma)}s_{\lambda^-(\gamma)}$. While it is true that the two Schur functions here are the images under $\lim_{N \to \infty} \pi_{w_N}$ of the two factors in Definition~\ref{defn:wyser-yong-schubert}, in general $\pi_{w_N}$ is not a ring homomorphism, so we must work a little harder.

\begin{lem}[\cite{M2}, equation (3.10)] \label{lem:isobaric} Suppose $k$ is such that $\phi_k \neq \phi_{k'}$ for all $k' \neq k$. Then
    \begin{equation*}
        \pi_{\phi_k} s_{\lambda}(X_{\phi_1}, \ldots, X_{\phi_k}, \ldots, X_{\phi_{\ell}}) = s_{\lambda}(X_{\phi_1}, \ldots, X_{\phi_k+1}, \ldots, X_{\phi_{\ell}}).
    \end{equation*}
    If $i \notin \{\phi_1, \ldots \phi_\ell\}$, then $\pi_i s_{\lambda}(X_\phi) = s_{\lambda}(X_\phi)$.
\end{lem}

\begin{proof}
    Let us first verify this when $\lambda = (d)$ has length $1$, so $s_{\lambda}(X_r)$ is the homogeneous symmetric polynomial $h_d(X_r) = h_d(x_1, \ldots, x_r)$, and we must see that $\pi_r h_d(X_r) = h_d(X_{r+1})$. This is easy using the generating function
    \begin{equation*}
        \prod_{i=1}^r \frac{1}{1 - x_i t} = \sum_{d=0}^{\infty} h_d(X_r) t^i.
    \end{equation*}
    The first $r-1$ factors on the left are symmetric in $x_r$ and $x_{r+1}$, so commute with $\pi_r$, so one only needs to verify by direct computation that $\pi_r (1-x_r t)^{-1} = (1 - x_r t)^{-1} (1 - x_{r+1} t)^{-1}$.

    For general $\lambda$, we use the Jacobi-Trudi identity for flagged Schur functions \cite{wachs-schubert}:
    \begin{equation*}
        s_{\lambda}(X_{\phi}) = \det\,(h_{\lambda_i-i+j}(X_{\phi_i}))_{1 \leq i,j \leq \ell(\lambda)}.
    \end{equation*}
    This determinant expands as a sum of terms of the form
    \begin{equation} \label{eq:det-term}
        \pm h_{d_1}(X_{\phi_1}) \cdots h_{d_\ell}(X_{\phi_\ell}).
    \end{equation}
    If $i \neq r$, then $h_d(X_r)$ is symmetric in $x_i$ and $x_{i+1}$. In particular, the hypothesis $\phi_{k'} \neq \phi_{k}$ for $k' \neq k$ ensures that every factor in the term \eqref{eq:det-term} is symmetric in $x_{\phi_k}$ and $x_{\phi_k+1}$ except for $h_{d_k}(X_{\phi_k})$. The effect of applying $\pi_{\phi_k}$ to the term \eqref{eq:det-term} is therefore the same as the effect of applying it only to the factor $h_{d_k}(X_{\phi_k})$, and the previous paragraph shows that this is the same as replacing $\phi_k$ by $\phi_{k}+1$. This argument also shows that if $i \notin \{\phi_1, \ldots, \phi_\ell\}$, then $s_{\lambda}(X_\phi)$ is symmetric in $x_i$ and $x_{i+1}$, hence fixed by $\pi_i$.
\end{proof}

\begin{thm} \label{thm:matchless-clan-stanley} $F_\gamma = s_{\lambda^+(\gamma)} s_{\lambda^-(\gamma)}$ for a matchless clan $\gamma$. \end{thm}
\begin{proof} Abbreviate $\lambda^\pm(\gamma)$ and $\phi^\pm(\gamma)$ as $\lambda^\pm$ and $\phi^\pm$. By Lemma~\ref{lem:stabilization} and the formulas of Definition~\ref{defn:wyser-yong-schubert},
    \begin{equation*}
        F_\gamma = \lim_{N\to \infty} \pi_{w_N}(s_{\lambda^+}(X_{\phi^+})\, s_{\lambda^-}(X_{\phi^-})).
    \end{equation*}
    Fix $N \geq n = p+q$. Let $a^i$ be the word $\bf i(i{+}1)\cdots (N{-}1)$ for $i < N$. It is not hard to check that $a^{N-1} \cdots a^2 a^1$ is a reduced word for $w_N$, and we will take $\pi_{w_N}$ to be the specific composition $\pi_{a^{N-1}} \cdots \pi_{a^1}$. Let $f = s_{\lambda^+}(X_{\phi^+})\, s_{\lambda^-}(X_{\phi^-})$.
    
    First consider $\pi_{N-1}(f)$. If $N-1 > n$, then $f$ is symmetric in $x_{N-1}$ and $x_N$ (since these variables do not even appear), so $f$ is fixed by $\pi_{N-1}$. Otherwise, $N-1$ appears in exactly one of $\phi^-$ and $\phi^+$; say $(\phi^-)_k = N-1$. The sequences $\phi^+$ and $\phi^-$ are disjoint and have no repeated entries, so it follows from Lemma~\ref{lem:isobaric} that
    \begin{equation*}
        \pi_{N-1}(f) = s_{\lambda^+}(X_{\phi^+}) \pi_{N-1}(s_{\lambda^-}(X_{\phi^-})) = s_{\lambda^+}(X_{\phi^+}) s_{\lambda^-}(X_{\phi^-_1}, \ldots, X_{\phi^-_k + 1}, \ldots, X_{\phi^-_q}); 
    \end{equation*}
    in words, $\pi_{n-1}(f)$ is obtained from $f$ by incrementing the entry $N-1$ of $\phi^-$ to $N$. This does not alter any entries of $\phi^\pm$ which are less than $N-1$, and so the same argument shows that subsequently applying $\pi_{N-2}, \pi_{N-3}, \ldots, \pi_1$ (in that order) has the effect of incrementing every value in $\phi^-$ and $\phi^+$ which is less than $N$. That is, $\pi_{a^1}(f) = s_{\lambda^+}(X_{\uparrow \phi^+}) s_{\lambda^-}(X_{\uparrow \phi^-})$ where for a sequence $\phi$ we define $\uparrow\!\!\phi$ as the sequence with
    \begin{equation*}
        (\uparrow\!\!\phi)_i = \begin{cases}
            \phi_i + 1 & \text{if $\phi_i < N$}\\
            \phi_i & \text{if $\phi_i \geq N$}.
        \end{cases}
   \end{equation*}

   Similarly, consider the action of $\pi_{a^2}$. Ignoring entries equal to $N$, the flags $\uparrow\!\!\phi^+$ and $\uparrow\!\!\phi^-$ are still disjoint with no repeated entries, and so the argument of the last paragraph shows that
   \begin{equation*}
    \pi_{a^2}(\pi_{a^1}f) = \pi_{a^2}(s_{\lambda^+}(X_{\uparrow \phi^+}) s_{\lambda^-}(X_{\uparrow \phi^-})) = s_{\lambda^+}(X_{\uparrow \uparrow \phi^+}) s_{\lambda^-}(X_{\uparrow \uparrow \phi^-}).
   \end{equation*}
   Continuing in this way, we see that
   \begin{equation*}
    \pi_{w_N} f = s_{\lambda^+}(X_{\uparrow^{N-1} \phi^+}) s_{\lambda^-}(X_{\uparrow^{N-1} \phi^-})) = s_{\lambda^+}(X_N, \ldots, X_N) s_{\lambda^-}(X_N, \ldots, X_N).
   \end{equation*}
   Since $\lambda^-$ and $\lambda^+$ have length at most $n \leq N$ by definition, $s_{\lambda^\pm}(X_N, \ldots, X_N)$ is simply the ordinary Schur polynomial $s_{\lambda^\pm}(x_1, \ldots, x_N)$. Thus, $\lim_{N \to \infty} \pi_{w_N} f = s_{\lambda^+} s_{\lambda^-}$.
        
\end{proof}  

\begin{cor} \label{cor:clan-word-count} Let $f^\lambda$ be the number of standard tableaux of shape $\lambda$. Then for a matchless clan $\gamma \in \Clan_{p,q}$,
    \begin{equation*}
        \#\Red(\gamma) = {|\lambda^+|+|\lambda^-| \choose |\lambda^+|,|\lambda^-|} f^{\lambda^+} f^{\lambda^-} = (pq)! \prod_{\substack{i \in \phi^+ \\ j \in \phi^-}} \frac{1}{|i-j|}. 
    \end{equation*}
\end{cor}

\begin{proof} The first equality follows from Theorem~\ref{thm:matchless-clan-stanley} by comparing coefficients of $x_1 x_2 \cdots $, as per Proposition~\ref{prop:reduced-word-count}. For the second, apply the hook length formula. The hook lengths of $\lambda^+$ are exactly the distances from each $+$ in $\gamma$ to some following $-$. To be precise, the hook with corner $(i,j)$ in $\lambda^+$ has size $\phi^+_i - \phi^-_{q-j+1}$. This statement and the corresponding statement for $\lambda^-$ imply via the hook length formula that
    \begin{equation*}
        f^{\lambda^+} f^{\lambda^-} = |\lambda^+|!\prod_{\substack{i \in \phi^+ \\ j \in \phi^- \\ i < j}} \frac{1}{j-i}\,\, |\lambda^-|!\prod_{\substack{i \in \phi^+ \\ j \in \phi^- \\ i > j}} \frac{1}{i-j}.
    \end{equation*}
    Since $|\lambda^+| + |\lambda^-|$ is the total number of pairs of a $+$ and following $-$ or vice versa, i.e. $pq$, the second equality follows.
\end{proof}

\section{Maximizing the number of reduced words} 
\label{sec:max}
 
Any reduced word for a permutation in $S_n$ is a prefix of at least one reduced word for $w_n$, so $\#\Red(w)$ is maximized when $w = w_n$. For the same reason, the clan $\gamma \in \Clan_{p,q}$ with the most reduced words must be matchless, but it is not immediately clear which matchless clans maximize $\#\Red(\gamma)$. In the smallest case $q \geq p = 1$, Corollary~\ref{cor:clan-word-count} says that $\#\Red(\gamma) = {q \choose \phi_1^+ - 1}$, confirming the natural guess that $\#\Red(\gamma)$ is maximized when $\gamma$ has its single $+$ as close to the middle of its one-line notation as possible. To proceed further, it is helpful to rewrite the formula of Corollary~\ref{cor:clan-word-count}.

\begin{prop} \label{prop:reduced-word-function} For real numbers $1 \leq \phi_1 < \cdots < \phi_m \leq p+q = n$ let
    \begin{equation*}
        f(\phi_1, \ldots, \phi_m) = \prod_{1 \leq k < \ell \leq m} \frac{1}{(\phi_k - \phi_\ell)^2} \prod_{k=1}^m \Gamma(\phi_k)\Gamma(n+1-\phi_k).
    \end{equation*}
         If $\gamma \in \Clan_{p,q}$ is matchless, then $\#\Red(\gamma) = (pq)! / f(\phi^+(\gamma)) = (pq)! / f(\phi^-(\gamma))$.
\end{prop}

\begin{proof}
    Regroup the factors in Corollary~\ref{cor:clan-word-count}:
    \begin{equation*}
        \#\Red(\gamma) = (pq)!\prod_{i \in \phi^+}\left[ \prod_{\substack{j \in \phi^- \\ j < i}} \frac{1}{i-j} \prod_{\substack{j \in \phi^- \\ j > i}} \frac{1}{j-i}\right].
    \end{equation*}
    Rewriting
    \begin{equation*}
        \prod_{\substack{j \in \phi^- \\ j < i}} \frac{1}{i-j} = \frac{1}{(i-1)!}\prod_{\substack{k \in \phi^+ \\ k < i}} (i-k) \qquad \text{and} \qquad \prod_{\substack{j \in \phi^- \\ j > i}} \frac{1}{j-i} = \frac{1}{(n-i)!} \prod_{\substack{k \in \phi^+ \\ k > i}} (k-i)
    \end{equation*}
    gives
    \begin{equation*}
        \#\Red(\gamma) = (pq)!\prod_{i \in \phi^+} \left[ \frac{1}{(i-1)!(n-i)!} \prod_{\substack{k \in \phi^+ \\ k \neq i}} (k-i)^2 \right] = \frac{(pq)!}{f(\phi^+)}.
    \end{equation*}
    The same argument works with the roles of $\phi^+$ and $\phi^-$ reversed. 
\end{proof}

Although there is not a unique maximizer of $\#\Red(\gamma)$, as for instance $\#\Red(+-++) = \#\Red(++-+)$, the next lemma provides a weaker uniqueness statement.

\begin{lem} \label{lem:unique-min}
    On the domain $1 \leq \phi_1 < \cdots < \phi_p \leq p+q = n$, the function $\log f(\phi_1, \ldots, \phi_p)$ is strictly convex in each variable. In particular, $f$ has a unique global minimum $\phi^* = (\phi_1^*, \ldots, \phi_p^*)$, and any minimizer of $f$ restricted to the integer lattice $\ZZ^p$ is one of the $2^p$ points obtained from $\phi^*$ by rounding each coordinate either up or down.
\end{lem}

\begin{proof} Fixing $\phi_2, \ldots, \phi_p$, we have
    \begin{equation*}
        \log f(\phi_1, \ldots, \phi_p) = \log \Gamma(\phi_1) + \log \Gamma(n+1-\phi_1) - 2\sum_{k=1} \log(\phi_k - \phi_1) + C
    \end{equation*}
    for some constant $C$. By the Bohr-Mollerup theorem, $\log \Gamma(\phi_1)$ is a convex function of $\phi_1$, and taking second derivatives shows the same is true of $\log \Gamma(n+1-\phi_1)$ and $-\log(\phi_k - \phi_1)$. In fact, $-\log(\phi_k - \phi_1)$ is strictly convex, so the sum $\log f(\phi_1, \ldots, \phi_p)$ is also strictly convex in $\phi_1$, and in every $\phi_i$ by symmetry.

    Strict convexity implies that $\log f$ (hence $f$) has at most one global minimum. To see that it does have one, observe that $f(\phi) \to \infty$ as $\phi$ approaches the boundary of the domain of $f$ where $\phi_i = \phi_{i+1}$ for some $i$, so that for sufficiently small $\epsilon > 0$, the global minimum of $f$ on the compact set where $1 \leq \phi_i \leq \phi_{i+1}-\epsilon \leq n$ for each $i$ will also be a global minimum of $f$ on its whole domain.

    Finally, using the convexity of $\log f$ in each variable individually, the claim about the minimizer of $\log f$ restricted to $\ZZ^n$ reduces to the fact that if $g : [a,b] \to \RR$ is a convex function with global minimum $x^*$, then $g$ is decreasing on $[a,x^*]$ and increasing on $[x^*, b]$.
\end{proof}

We can work out almost exactly which $(2,q)$-clans maximize $\#\Red(q)$. The minimizer $\phi^*$ of $f$ from Lemma~\ref{lem:unique-min} must be invariant under the transformation
\begin{equation*}
    (\phi_1^*, \ldots, \phi_p^*) \mapsto (n+1-\phi_p^*, \ldots, n+1-\phi_1^*),
\end{equation*}
since $f$ itself is. When $\phi = \phi^+(\gamma)$ this transformation corresponds to reversing the one-line notation of $\gamma$. In particular, $\phi^*$ is determined by the one parameter $\phi_1^*$ for $p=2$, or equivalently by the distance between $\phi_1^*$ and $\phi_2^*$. The next theorem shows that $\#\Red(\gamma)$ is maximized for $\gamma \in \Clan_{2,q}$ when $\gamma$ is (essentially) invariant under reversal and its two $+$ signs are separated by distance $\sqrt{n}$.

\begin{thm} When $p=2$, the minimizer $\phi^* = (\phi_1^*, \phi_2^*)$ of $f$ satisfies
    \begin{equation} \label{eq:minimizer-bound}
        \left|\phi_1^* - \left(\frac{n+1}{2} - \frac{1}{2}\sqrt{n}\right)\right| \leq \frac{33}{64}.
    \end{equation}
    Setting $\alpha_1 = \frac{n+1}{2} - \frac{1}{2}\sqrt{n}$ and $\alpha_2 = \frac{n+1}{2} + \frac{1}{2}\sqrt{n}$, the clans $\gamma \in \Clan_{2,q}$ maximizing $\#\Red(\gamma)$ have
    \begin{align*}
        &\phi^+(\gamma)_1 \in \{\lfloor \alpha_1 \rfloor-1, \lfloor \alpha_1 \rfloor, \lceil \alpha_1 \rceil, \lceil \alpha_1 \rceil + 1\}\\
        &\phi^+(\gamma)_2 \in \{\lfloor \alpha_2 \rfloor-1, \lfloor \alpha_2 \rfloor, \lceil \alpha_2 \rceil, \lceil \alpha_2 \rceil + 1\}.
    \end{align*}
\end{thm}

\begin{proof} Lemma~\ref{lem:unique-min} shows that $\phi^+(\gamma)_1$ is one of the two closest integers to $\phi_1^*$, so if it is known that $|\phi_1^* - \alpha_1| < 1$, then $\phi^+(\gamma)_1$ must be one of $\lfloor \alpha_1 \rfloor-1, \lfloor \alpha_1 \rfloor, \lceil \alpha_1 \rceil, \lceil \alpha_1 \rceil + 1$. The analogous fact for $\phi^+(\gamma)_2$ holds by symmetry of $\phi^*$. Thus, it suffices to prove the bound \eqref{eq:minimizer-bound}.

    Since $\phi_2^* = n+1-\phi_1^*$, we may as well minimize the single variable function $f(\phi_1, n+1-\phi_1)$ on the domain $\phi_1 \in [1, \frac{n}{2}]$. It is helpful to let $m = \frac{n+1}{2}$ and use the new coordinate $x = m-\phi_1$:
    \begin{align*}
        \log f(\phi_1, n+1-\phi_1) &= 2\log \Gamma(n+1-\phi_1) + 2\log \Gamma(\phi_1) - 2\log(n+1-2\phi_1)\\
        &= 2\log \Gamma(m+x) + 2\log \Gamma(m-x) - 2\log(2x).
    \end{align*}
    Set $g(x) = \log \Gamma(m+x) + \log \Gamma(m-x) - \log(2x)$. Then
    \begin{equation*}
        g'(x) = \Psi(m+x) - \Psi(m-x) - \frac{1}{x},
    \end{equation*}
    where $\Psi(y) = \frac{d}{dy} \log \Gamma(y)$. The inequalities $\log(y-\frac{1}{2}) < \Psi(y) < \log(y)$ for $y > \frac{1}{2}$ and $\log(1+y) \leq y$ for $y > -1$ imply
    \begin{align*}
        -\log \frac{m-x}{m+x-{\textstyle \frac{1}{2}}} - \frac{1}{x} <\, &g'(x) < \log \frac{m+x}{m-x-{\textstyle \frac{1}{2}}} - \frac{1}{x}\\
        \Longrightarrow \quad \frac{2x - {\textstyle \frac{1}{2}}}{m+x-{\textstyle \frac{1}{2}}} - \frac{1}{x} <\,  &g'(x) < \frac{2x + {\textstyle \frac{1}{2}}}{m-x-{\textstyle \frac{1}{2}}} - \frac{1}{x}.
    \end{align*}
    The positive zeroes of the lower and upper bounds here are, respectively, 
    \begin{equation*}
        \frac{3}{8} + \frac{1}{2}\sqrt{2m - \frac{7}{16}} \quad \text{and} \quad -\frac{3}{8} + \frac{1}{2}\sqrt{2m - \frac{7}{16}},
    \end{equation*}
    or $\pm \frac{3}{8} + \frac{1}{2} \sqrt{n + \frac{9}{16}}$. It follows that $g$ has a critical point $x^*$ in $(-\frac{3}{8} + \frac{1}{2} \sqrt{n + \frac{9}{16}}, \frac{3}{8} + \frac{1}{2} \sqrt{n + \frac{9}{16}})$, and strict convexity of $g$ forces $x^*$ to be its unique global minimizer.

    Using these bounds on $x^*$ and the inequality $\sqrt{n+9/16}-\sqrt{n} = \frac{9/16}{\sqrt{n+9/16}+\sqrt{n}} \leq 9/32$ for $n \geq 1$ gives
    \begin{equation*}
        -\frac{3}{8} < x^* - \frac{1}{2}\sqrt{n} < \frac{3}{8} + \frac{9}{64} = \frac{33}{64}.
    \end{equation*}
    Since $\phi_1^* = \frac{n+1}{2} - x^*$, the bound \eqref{eq:minimizer-bound} follows.
\end{proof}

Let $\nred(\gamma)$ be the number of reduced words of $\gamma$. Although we do not have a general description of the clans maximizing $\nred$, we can prove a sort of continuity result showing that a maximizer of $\nred$ in $\Clan_{p,q}$ cannot be very different from a maximizer in $\Clan_{p,q+1}$. Define a partial order $\preceq$ on matchless $(p,q)$-clans by declaring $\gamma' \preceq \gamma$ if $\phi^+(\gamma')_i \leq \phi^+(\gamma)_i$ for $i = 1, \ldots, p$. This partial order is a lattice, with
\begin{equation*}
    \phi^+(\gamma \vee \gamma')_i = \max(\phi^+(\gamma)_i, \phi^+(\gamma)_i) \quad \text{and} \quad \phi^+(\gamma \wedge \gamma')_i = \min(\phi^+(\gamma)_i, \phi^+(\gamma)_i).
\end{equation*}
Write $\gamma-$ and $-\gamma$ for the clans obtained by appending or prepending a $-$ to the one-line notation of $\gamma$.

\begin{lem} \label{lem:reduced-word-comparison}
    If $\gamma' \prec \gamma$ and $\nred(\gamma') \leq \nred(\gamma)$, then $\nred(\gamma'-) < \nred(\gamma-)$.
\end{lem}

\begin{proof}
    Abbreviate $\phi^+(\gamma)$ and $\phi^+(\gamma')$ as $\phi$ and $\phi'$. Proposition~\ref{prop:reduced-word-function} shows
    \begin{equation*}
        \nred(\gamma) = (pq)!\prod_{1 \leq i < j \leq p}(\phi_j - \phi_i)^2 \prod_{i=1}^p \frac{1}{(\phi_i-1)!(n-\phi_i)!}.
    \end{equation*}
    Thus,
    \begin{equation*}
        \frac{\nred(\gamma-)}{\nred(\gamma'-)} = \prod_{i=1}^p \frac{(n+1-\phi_i')!}{(n+1-\phi_i)!} = \frac{\nred(\gamma)}{\nred(\gamma')} \prod_{i=1}^p \frac{n-\phi_i'+1}{n-\phi_i+1},
    \end{equation*}
    and the last expression strictly exceeds $\nred(\gamma)/\nred(\gamma') \geq 1$ because $\gamma' \prec \gamma$.
\end{proof}

\begin{thm} Suppose $\gamma \in \Clan_{p,q}$ and $\delta \in \Clan_{p,q+1}$ are maximizers of $R$. Then all entries of the vector $\phi^+(\delta) - \phi^+(\gamma)$ are either $0$ or $1$. \end{thm}
    \begin{proof}
        Suppose $\epsilon \in \Clan_{p,q+1}$ is such that $\gamma- \not\preceq \epsilon$. We will show that $\epsilon$ then does not maximize $\nred$, so that necessarily $\gamma- \preceq \delta$ and (by a symmetric argument) $-\gamma \preceq \delta$, which together imply the theorem.

        It is clear that the one-line notation of $\epsilon \wedge \gamma-$ ends in $-$, so let $\zeta \in \Clan_{p,q}$ be such that ${\zeta-} = \epsilon \wedge \gamma-$. Then $\nred(\zeta) \leq \nred(\gamma)$ by the choice of $\gamma$, and the (strict!) inequality $\epsilon \wedge {\gamma-} \prec {\gamma-}$ implies $\zeta \prec \gamma$. Therefore $\nred(\zeta-) = \nred(\epsilon \wedge \gamma-) < \nred(\gamma-)$ by Lemma~\ref{lem:reduced-word-comparison}.

        Now, using the formula of Proposition~\ref{prop:reduced-word-function},
        \begin{equation*} 
            \frac{\nred(\epsilon \wedge \gamma-)\nred(\epsilon \vee \gamma-)}{\nred(\gamma-)\nred(\epsilon)} = \left[\prod_{1 \leq i < j \leq p} \frac{\phi^+_j(\epsilon \wedge \gamma-) - \phi^+_i(\epsilon \wedge \gamma-) }{\phi^+_j(\gamma-) - \phi^+_i(\gamma-) } \frac{\phi^+_j(\epsilon \vee \gamma-) - \phi^+_i(\epsilon \vee \gamma-)}{\phi^+_j(\epsilon) - \phi^+_i(\epsilon)} \right]^2,
        \end{equation*}
        which is at least $1$ by the general inequality
        \begin{equation*}
            (a_2-a_1)(b_2-b_1) \leq (\min(a_2,b_2)-\min(a_1,b_1))(\max(a_2,b_2)-\max(a_1,b_1))
        \end{equation*}
        when $a_1 < a_2$ and $b_1 < b_2$. Having shown $\nred(\epsilon \wedge \gamma-) < \nred(\gamma-)$ in the previous paragraph, this implies $\nred(\epsilon \vee \gamma-) > \nred(\epsilon)$, so $\epsilon$ does not maximize $\nred$.
    \end{proof}

    We conclude this section by describing connections to work of Pittel and Romik on random Young tableaux of rectangular shape, although we do not attempt to prove any precise results. The uniqueness in Lemma~\ref{lem:unique-min} shows that if $\gamma^* \in \Clan_{p,q}$ maximizes $\nred(\gamma^*)$, then $\gamma^*$ and $\rev(\gamma^*)$ should be effectively equal (to be precise, $\phi^+(\gamma^*)$ and $\phi^+(\rev(\gamma^*))$ differ by a vector with entries from $\{0,1,-1\}$). Proposition~\ref{prop:symmetries} then implies $\lambda^+(\gamma^*) \approx \lambda^+(\gamma^*)^{\vee} = \lambda^-(\gamma^*)^t$, so
    \begin{equation*}
        \nred(\gamma^*) \approx {pq \choose pq/2} f^{\lambda^+(\gamma^*)} f^{\lambda^-(\gamma^*)} = {pq \choose pq/2} f^{\lambda^+(\gamma^*)} f^{(\lambda^+(\gamma^*)^\vee)^t}  = {pq \choose pq/2} f^{\lambda^+(\gamma^*)} f^{\lambda^+(\gamma^*)^\vee}.
    \end{equation*}
    Thus, maximizing $\nred$ is equivalent to maximizing $f^{\lambda}f^{\lambda^\vee}$ over $\lambda \subseteq [p] \times [q]$ with $|\lambda| \approx \lfloor pq/2 \rfloor$.

    Let $\SYT(\lambda)$ be the set of standard tableaux of shape $\lambda$, and write $(q^p)$ for the $p \times q$ rectangular partition. For any fixed $0 \leq k \leq pq$, there is a bijection
    \begin{equation*}
        \SYT(q^p) \to \bigcup_{\substack{\lambda \vdash k \\ \lambda \subseteq [p] \times [q]}} \SYT(\lambda) \times \SYT(\lambda^{\vee}),
    \end{equation*}
    which sends $T \in \SYT(q^p)$ to $(T_1, T_2)$ where $T_1$ is the subtableau of $T$ containing $1, 2, \ldots, k$, and $T_2$ is the complement of $T_1$ in $T$ rotated $180^\circ$ and with the entries $pq, pq-1, \ldots, k+1$ replaced by $1, 2, \ldots, pq-k$. It follows that $f^{\lambda} f^{\lambda^\vee} / f^{(q^p)}$ is the probability that the entries in $[|\lambda|]$ of a uniformly random member of $\SYT(q^p)$ form a subtableau of shape $\lambda$. By the previous paragraph we would like to maximize this probability over $\lambda$ with $|\lambda| = \lfloor pq/2 \rfloor$.

    In \cite{pittel-romik}, Pittel and Romik describe a ``typical'' random standard tableau of shape $(q^p)$ when $p, q$ are large (and in a fixed ratio). To be precise, given $T \in \SYT(q^p)$ let $S_T : [0,1) \times [0,p/q)$ be the function 
    \begin{equation*}
        S_T(x,y) = \frac{1}{pq}T(\lfloor qy \rfloor+1, \lfloor qx \rfloor+1),
    \end{equation*}
    where $T(i,j)$ is the entry of $T$ in row $i$ and column $j$. That is, we think of $T$ as a surface whose height above the $xy$-plane is given by the entries of $T$, rescaled so that the maximum height is $1$ and the surface lies above the rectangle $[0,1) \times [0,p/q)$. It is helpful to picture $T$ in the French style here, so that $1$ is in its lower-left corner at $(0,0)$ and $pq$ is in its upper-right corner.

    \begin{thm}[\cite{pittel-romik}, Theorem 5] Fix $\theta \in (0,1]$, and suppose $p_1, p_2, \ldots$ is a sequence of integers such that $\lim_{q \to \infty} p_q/q = \theta$. There is an explicit function $L_{\theta} : [0,1] \times [0,\theta] \to [0,1]$ such that for all $\epsilon > 0$ and all $(x,y) \in [0,1) \times [0,\theta)$,
        \begin{equation*}
            \lim_{q\to \infty} \mathbf{P}(|S_T(x,y) - L_{\theta}(x,y)| > \epsilon \,:\, \text{$T \in \SYT(q^{p_q})$ uniformly random} ) = 0.
        \end{equation*}
    \end{thm}

    In particular, for large $q$, a random $T \in \SYT(q^{p_q})$ has its entries $1, 2, \ldots, \lfloor p_q q/2 \rfloor$ contained in a subtableau whose shape resembles the region in $[0,1) \times [0,\theta)$ below the level curve $\{(x,y) : L_{\theta}(x,y) = \frac{1}{2}\}$. We expect that if a matchless clan $\gamma$ is chosen as the top element of a uniformly random maximal chain in $\Clan_{p_q,q}$ with $q$ large, then $\lambda^+(\gamma)$ should resemble this same limiting shape, and $\gamma$ should be close to a maximizer of $\nred$ with high probability.
    
    It is natural to describe the resulting ``limit clan'' by a density function $f : [0,1] \to \RR$, so that for $t \in [0,1]$,
    \begin{equation}
        \text{\# of $+$'s among $\gamma_1, \gamma_2, \ldots, \gamma_{\lfloor t(p+q) \rfloor}$} \approx p\int_0^t f(t')\,dt'. \label{eq:cumulative}
    \end{equation}
    Write $C(t) = \int_0^t f(t')\,dt'$. If \eqref{eq:cumulative} holds, then
    \begin{equation} \label{eq:cumulative-2}
        C\left(\frac{\phi^+_i}{p+q}\right) = C\left(\frac{q-\lambda^+_i+i}{p+q}\right)  \approx \frac{i}{p}
    \end{equation}
    for $i \in [p]$, by definition of $\phi^+(\gamma)$. Letting $p,q \to \infty$ (with $p/q \to \theta$) and replacing $i/p$ with $t \in [0,1]$, equation \eqref{eq:cumulative-2} becomes
    \begin{equation*}
        C\left(\frac{1-x(t)+\theta t}{1+\theta}\right) = t. 
    \end{equation*}
    where $x(t)$ is such that $L_{\theta}(x(t), \theta t) = \frac{1}{2}$. Using the explicit formulas from \cite{pittel-romik}, one finds
    \begin{equation*}
        C'(t) = f(t) = \begin{cases}
            \frac{1+\theta}{2\theta} \left[1 - \frac{2}{\pi} \sin^{-1}\left(\frac{1-\theta}{1+\theta}\frac{1}{2\sqrt{t(1-t)}}\right)\right] & \text{if $|t-\frac{1}{2}| < \frac{\sqrt{\theta}}{\theta+1}$}\\
            0 & \text{otherwise}
        \end{cases}
    \end{equation*}

    

\section{Connections to involution words}
\label{sec:inv-vs-clan} 

Let $\I_n$ be the set of involutions in $S_n$. Given $z \in \I_n$ and an adjacent transposition $s_i$, define
\begin{equation*}
    z \ast s_i = \begin{cases}
        z s_i & \text{if $s_i z = z s_i$}\\
        s_i z s_i & \text{otherwise}
    \end{cases}
\end{equation*}
Note that $z \ast s_i$ is again an involution. The \emph{weak order} on $\I_n$ is the transitive closure of the relations $z \ast s_i < z$ when $\ell(z \ast s_i) < \ell(z)$ \cite{CJW, HMP1, HMP2, hultman-twisted-involutions, richardson-springer}.

\begin{defn} A \emph{reduced involution word} for $z \in \I_n$ is the sequence of labels along a saturated chain in weak order from the identity involution to $z$. Equivalently, it is a minimal-length word $a_1 \cdots a_\ell$ such that
    \begin{equation*}
        z = (\cdots ((1 \ast s_{a_1})  \ast s_{a_2}) \ast \cdots ) \ast s_{a_\ell}.
    \end{equation*}
    To avoid confusion with usual reduced words for $z$, we write $\hat \Red(z)$ for the set of reduced involution words of $z$.
\end{defn}

\begin{ex}
    The weak order on $\I_3$, with involutions drawn as partial matchings of $\{1,2,3\}$:
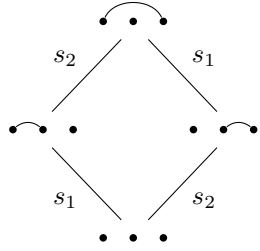
    \begin{center}
        \begin{tikzpicture}[scale=1.2, auto]
            \node[circle, fill, inner sep=0pt, minimum size = 1mm] (858286519) at (0.0,0) {};
            \node[circle, fill, inner sep=0pt, minimum size = 1mm] (858286520) at (0.333,0) {};
            \node[circle, fill, inner sep=0pt, minimum size = 1mm] (858286521) at (0.667,0) {};
            
            \node[circle, fill, inner sep=0pt, minimum size = 1mm] (495822353) at (-1.0,1.2) {};
            \node[circle, fill, inner sep=0pt, minimum size = 1mm] (495822354) at (-0.667,1.2) {};
            \node[circle, fill, inner sep=0pt, minimum size = 1mm] (495822355) at (-0.333,1.2) {};
            \draw (495822353) to[bend left=40] (495822354);
            
            \node[circle, fill, inner sep=0pt, minimum size = 1mm] (385313341) at (1.0,1.2) {};
            \node[circle, fill, inner sep=0pt, minimum size = 1mm] (385313342) at (1.333,1.2) {};
            \node[circle, fill, inner sep=0pt, minimum size = 1mm] (385313343) at (1.667,1.2) {};
            \draw (385313342) to[bend left=40] (385313343);
            
            \node[circle, fill, inner sep=0pt, minimum size = 1mm] (336065196) at (0.0,2.4) {};
            \node[circle, fill, inner sep=0pt, minimum size = 1mm] (336065197) at (0.333,2.4) {};
            \node[circle, fill, inner sep=0pt, minimum size = 1mm] (336065198) at (0.667,2.4) {};
            \draw (336065196) to[bend left=65] (336065198);

            \draw (0.2,0.2) to node {$s_1$} (-.566,1.0);
            \draw (0.5,0.2) to node[swap] {$s_2$} (1.26,1.0);
            \draw (-.566,1.4) to node {$s_2$} (0.2,2.2);
            \draw (1.26,1.4) to node[swap] {$s_1$} (0.5,2.2);
        \end{tikzpicture}
    \end{center}
    The reduced involution words of the maximal element $(1\,3)$ are $\bf 12$ and $\bf 21$.
\end{ex}

Just as in weak Bruhat order on $S_n$, the operation $z \mapsto z \ast s_i$ moves up or down in involution order according to whether $i$ is an ascent or descent of $z$.
\begin{prop}[\cite{hultman-twisted-involutions}, Lemma 3.8] \label{prop:ascent-descent} If $z(i) > z(i+1)$ then $z \ast s_i < z$, and if $z(i) < z(i+1)$ then $z \ast s_i > z$. \end{prop}

Write $\kappa(z)$ for the number of 2-cycles in an involution $z$, and define $\I_{p,q} = \{z \in \I_{p+q} : \kappa(z) \leq \min(p,q)\}$. Let $w_{p,q}$ be the involution $(1\,\,n)(2\,\,n{-}1)\cdots (m\,\,n{-}m{+}1)$ where $m = \min(p,q)$, so $w_{p,q} = \iota(\gamma_{p,q})$.
\begin{lem} \label{lem:max-element}
    The set $\I_{p,q}$ has $w_{p,q}$ as its unique maximal element in involution weak order.
\end{lem}

\begin{proof}
    Set $m = \min(p,q)$. First we prove the lemma with $\I_{p,q} = \{z \in \I_n : \kappa(z) \leq m\}$ replaced by $\{z \in \I_n : \kappa(z) = m\}$; call the latter set $J$. Suppose $z$ is maximal in $J$. By Proposition~\ref{prop:ascent-descent} this is equivalent to the condition that if $z(i) < z(i+1)$, then $z \ast s_i \notin J$, which can only happen if $z \ast s_i$ has $\kappa(z)+1$ cycles, i.e. if $z(i) = i$ and $z(i+1) = i+1$. Letting $i$ and $j$ be such that
    \begin{equation*}
        z(1) > \cdots > z(i-1) > z(i) < z(i+1) < \cdots < z(j-1) > z(j),
    \end{equation*}
    it follows that $i, i+1, \ldots, j-1$ are all fixed points.  We have $z(j) < z(j-1) = j-1$, and $z(j)$ is none of $\{z(j-1), \ldots, z(i)\} = \{j-1, \ldots, i\}$, so it must be one of $i-1, \ldots, 2, 1$ (assuming $z(j)$ exists). But the one-line notation of $z$ must end with $(i-1)\cdots 21$: otherwise, $z$ would have an ascent beginning with one of $1, 2, \ldots, i-1$, which would contradict maximality of $z$ since those are not fixed points. This completely determines $z$:
    \begin{align*}
        z &= n(n{-}1)\cdots (n{-}i{+}2)i(i{+}1)\cdots (n{-}i{+}1)(i{-}1)\cdots 21\\
        &= (1\,\,n)(2\,\,n{-}1)\cdots (i{-}1\,\,n{-}i{+}2)\\
        &= (1\,\,n)(2\,\,n{-}1)\cdots (m\,\,n{-}m{+}1) = w_{p,q} \qquad \text{(given that $\kappa(z) = m$)}
    \end{align*}

    Now let us see that $w_{p,q}$ is also the unique maximal element of $\{z \in \I_n : \kappa(z) \leq m\}$. By induction on $m$, we can assume $y := (1\,\,n)(2\,\,n-1)\cdots (m{-}1\,\,n{-}m{+}2)$ is the unique maximal element of $\{z \in \I_n : \kappa(z) \leq m-1\}$, and it is enough to show that $y < w_{p,q}$. But $y \ast s_m = (1\,\,n)(2\,\,n-1)\cdots (m{-}1\,\,n{-}m{+}2)(m\,\,m{+}1)$ is in $J$, so $y < y \ast s_m \leq w_{p,q}$ by the previous paragraph.
\end{proof}
 
\newpage

\begin{prop} \label{prop:clan-vs-inv} \hfill
    \begin{enumerate}[(a)]
        \item If $\gamma \ast s_i$ is defined, then $\iota(\gamma \ast s_i) = \iota(\gamma) \ast s_i$.
        \item If $\gamma \ast s_i < \gamma$, then $\iota(\gamma \ast s_i) > \iota(\gamma)$.
        \item Let $\gamma \in \Clan_{p,q}$ and suppose $i$ is a descent of $\iota(\gamma)$.
        \begin{itemize}
            \item If $\iota(\gamma) \ast s_i = s_i\iota(\gamma)s_i$, there is a unique $\gamma' > \gamma$ in $\Clan_{p,q}$ such that $\gamma' \ast s_i = \gamma$;
            \item If $\iota(\gamma) \ast s_i = \iota(\gamma)s_i$, there are exactly two such $\gamma'$. 
        \end{itemize}
        \item $\iota : \Clan_{p,q} \to \I_n$ is an order-reversing map with image $\I_{p,q}$.
     \end{enumerate}
\end{prop}

\begin{proof} \hfill
    \begin{enumerate}[(a)]
        \item The cases in which $\gamma \ast s_i$ is defined are: (1) if $i$ and $i+1$ are fixed points of $\gamma$ of opposite sign, then $\gamma \ast s_i$ is obtained from $\gamma$ by making $i$ and $i+1$ matched; (2) if $\{i,i+1\}$ is not $\iota(\gamma)$-invariant, then $\gamma \ast s_i = s_i \gamma s_i$. In case (1) $\iota(\gamma \ast s_i) = \iota(\gamma) s_i$ and $s_i$ commutes with $\iota(\gamma)$, while in case (2) $\iota(\gamma \ast s_i) = s_i\iota(\gamma)s_i \neq \iota(\gamma)$, so in either case we get $\iota(\gamma \ast s_i) = \iota(\gamma) \ast s_i$.

        \item The relation $\gamma \ast s_i < \gamma$ implies $\ell(\iota(\gamma \ast s_i)) > \ell(\iota(\gamma))$, and by part (a) this is the same as $\ell(\iota(\gamma) \ast s_i) > \ell(\iota(\gamma))$, which means $\iota(\gamma) \ast s_i > \iota(\gamma)$ in weak order on $\I_n$.

        \item Suppose $\gamma'$ is such that $\gamma' \ast s_i = \gamma$.  Because $i$ is a descent of $\iota(\gamma)$, Proposition~\ref{prop:ascent-descent} implies that $\ell(\iota(\gamma')) = \ell(\iota(\gamma) \ast s_i) < \ell(\iota(\gamma))$ (using part (a)), so that $\gamma' > \gamma$.
        
        If $\iota(\gamma) \ast s_i = s_i\iota(\gamma)s_i$ then $i, i+1$ are not matched by $\gamma$ and they are not both fixed points, so the same is true of $\gamma'$. In that case, $\gamma' \ast s_i$ is defined as $s_i \gamma' s_i$, forcing $\gamma' = s_i \gamma s_i$.
        
        If $\iota(\gamma) \ast s_i = \iota(\gamma)s_i$, then $i$ and $i+1$ are matched by $\gamma$ (they cannot be fixed points since $i$ is a descent of $\iota(\gamma)$). Thus $\gamma'$ and $\gamma$ agree on $[n] \setminus \{i,i+1\}$, and $i,i+1$ must be fixed points of $\gamma'$ labeled $-+$ or $+-$ in order to have $\gamma' \ast s_i = \gamma$.

        \item Part (b) shows that $\iota$ is order-reversing. An involution $z \in S_n$ has $n - 2\kappa(z)$ fixed points, and constructing $\gamma \in \Clan_{p,q}$ with $\iota(\gamma) = z$ is equivalent to choosing $a$ of those fixed points to label $+$ and $b$ of them to label $-$, subject to the constraints $a+b = n - 2\kappa(z)$ and $a-b = p-q$. This gives $a = p-\kappa(z)$ and $b = q-\kappa(z)$, so $z \in \iota(\Clan_{p,q})$ if and only if $\kappa(z) \leq \min(p,q)$. In fact, we get the stronger result that
        \begin{equation*}
            \#\{\gamma \in \Clan_{p,q} : \iota(\gamma) = z\} = {n - 2\kappa(z) \choose p-\kappa(z)} = {n - 2\kappa(z) \choose q-\kappa(z)}.
        \end{equation*}
        
        
    \end{enumerate}
\end{proof}

\begin{lem} \label{lem:clan-chains-vs-inv-chains} Suppose $C$ is a saturated chain $1 = z^0 < z^1 < z^2 < \cdots < z^r = z$ in $\I_{p,q}$.
\begin{enumerate}[(a)]
    \item  For a fixed $\gamma \in \Clan_{p,q}$ with $\iota(\gamma) = z$, there are exactly $2^{\kappa(z)}$ saturated chains in $\Clan_{p,q}$ with minimal element $\gamma$ whose image under $\iota$ is $C$.
\item  The total number of saturated chains in $\Clan_{p,q}$ with image $C$ is
\begin{equation*}
    {n - 2\kappa(z) \choose p-\kappa(z)}2^{\kappa(z)} = {n - 2\kappa(z) \choose q-\kappa(z)}2^{\kappa(z)}.
\end{equation*}

\end{enumerate}\end{lem}

\begin{proof}
    Part (b) follows from (a) because the number of $\gamma \in \Clan_{p,q}$ such that $\iota(\gamma) = z$ is ${n - 2\kappa(z) \choose p-\kappa(z)}$, as per the proof of Proposition~\ref{prop:clan-vs-inv}(d).

    As for part (a),  let $k$ be the number of covering relations $z^{j} < z^{j+1}$ in the chain $z^0 < z^1 < \cdots < z^r = z$ for which $z^{j+1} = z^{j}s_i$ for some $i$ (as opposed to $z^{j+1} = s_i z^j s_i$). Proposition~\ref{prop:clan-vs-inv}(c,a) show that the number of saturated chains in $\Clan_{p,q}$ with image $C$ and minimal element $\gamma$ is $2^k$. But the number $k$ is $\kappa(z)$ for any saturated chain from $1$ to $z$, because
    \begin{equation*}
        \kappa(z \ast s_i) = 
        \begin{cases}
            \kappa(z) & \text{if $z \ast s_i = s_i z s_i$}\\
            \kappa(z)+1 & \text{if $z \ast s_i = z s_i$ and $\ell(zs_i) > \ell(z)$}
        \end{cases}.
    \end{equation*}
\end{proof}
Because $\iota$ is order-reversing, Lemma~\ref{lem:clan-chains-vs-inv-chains} does not in general relate reduced words for $\gamma \in \Clan_{p,q}$ to reduced involution words for $\iota(\gamma)$. However, it does when $z = w_{p,q}$ is maximal in $\I_{p,q}$.

\begin{cor} The number of maximal chains in the poset $\Clan_{p,q}$ is $2^{\min(p,q)} \#\hat\Red(w_{p,q})$. \end{cor}

We can go further using known results for involution words.
\begin{defn}
    The \emph{involution Stanley symmetric function} of $z \in \I_n$ is
    \begin{equation*}
        \hat F_z = \lim_{m \to \infty} \sum_{a \in \bf\hat\Red(z^{+m})} \sum_{b \in \comp(a)} x_{b_1} \cdots x_{b_{\ell}}.
    \end{equation*}
\end{defn}
Just as for clans, the set $\hat \Red(z)$ is closed under the Coxeter relations for $S_n$ \cite[3.16]{richardson-springer}, so can be written as a disjoint union $\bigcup_{w \in \A(z)} \Red(w)$ over some set $\A(z) \subseteq S_n$. This implies that $\hat F_z = \sum_{w \in \A(z)} F_w$, so $\hat F_z$ is indeed a symmetric function.

\begin{defn} \label{def:shifted} A partition $\lambda$ is \emph{strict} if $\lambda_1 > \lambda_2 > \cdots > \lambda_{\ell}$, and the \emph{shifted shape} of a strict $\lambda$ is the set of boxes $\{(i,j) : \text{$1 \leq i \leq \ell(\lambda)$ and $i \leq j \leq i+\lambda_i-1$}\}$ in matrix coordinates. A filling of a shifted shape by the alphabet $\{1' < 1 < 2' < 2 < \cdots\}$ is a \emph{marked shifted semistandard tableau} if:
    \begin{itemize}
        \item Its entries are weakly increasing across rows and down columns;
        \item No unprimed (resp. primed) number appears twice in a column (resp. row);
        \item There are no primed numbers on the main diagonal.
    \end{itemize}
    The \emph{Schur P-function} of shifted shape $\lambda$ is $P_{\lambda} = \sum_{T} x^T$ where $T$ runs over marked shifted semistandard tableaux of shape $\lambda$. Here $x^T$ is the monomial in which the power of $x_i$ is the number of entries $i$ and $i'$ in $T$. The \emph{Schur Q-function} $Q_{\lambda}$ is then defined to be $2^{\ell(\lambda)}P_{\lambda}$. These are both symmetric functions \cite[III \sectionsymbol 8]{macdonald}.
\end{defn}

\begin{ex} \newcommand{\twoprime}{2'} \newcommand{\sixprime}{6'}
    Here is a marked shifted semistandard tableau of shifted shape $(5,4,1)$:
    \begin{equation*}
        \young(1\twoprime 22\sixprime,:245\sixprime,::5)
    \end{equation*}
\end{ex}

\begin{lem} \label{lem:dominant} Suppose $z = (1\,\,b_1)(2\,\,b_2)\cdots (k\,\,b_k)$ where $b_1 > \cdots > b_k > k$. Then $\hat F_z = P_{\lambda}$ where $\lambda = (b_1-1, b_2-2, \ldots, b_k-k)$.
\end{lem}

\begin{proof}
     For $y \in \I_n$, let $D(y) = \{(i,j) : \text{$j > i$, $z(j) < z(i)$, and $z(j) \leq i$}\}$, thought of as a subset of $[n] \times [n]$ in matrix coordinates. Let $\mu$ be the partition whose parts are the row lengths of $D(y)$. By \cite[Corollary 4.42]{HMP4}, $\hat F_y$ is a nonnegative integer combination of Schur $P$-functions, whose leading term in dominance order is $P_{\mu^t}$, where $\mu^t$ is the partition conjugate to $\mu$. For $z$ as defined above, one checks that $D(z)$ is the transpose of the shifted Young diagram of $\lambda = (b_1-1, \ldots, b_k-k)$, so the leading term of $\hat F_z$ is $P_{\lambda}$.

    Equivalently, the leading term of the Schur Q expansion of $2^{\kappa(z)}\hat F_z$ is $2^{\kappa(z)} P_{\lambda} = 2^k P_{\lambda} = Q_{\lambda}$. By \cite[Theorem 4.67]{HMP4}, $2^{\kappa(y)}\hat F_y$ equals a single Schur $Q$-function with coefficient $1$ if and only if $y$ is a $2143$-avoiding permutation, i.e. there are no $a < b < c < d$ such that $y(b) < y(a) < y(d) < y(c)$. This condition holds for $z$, so $2^{\kappa(z)}\hat F_z = Q_{\lambda}$, or $\hat F_z = P_{\lambda}$.
\end{proof}

\begin{thm} \label{thm:maximal-chain-count} Assume $p \geq q$ without loss of generality. Then
    \begin{equation*}
        \sum_{\substack{\gamma \in \Clan_{p,q} \\ \text{$\gamma$ matchless}} } F_\gamma = 2^q P_{(n-1,n-3,\ldots,n-2q+1)} = Q_{(n-1,n-3,\ldots,n-2q+1)}.
    \end{equation*}
    The number of maximal chains in $\Clan_{p,q}$ is
    \begin{equation*}
        2^{pq} (pq)!\left[\prod_{i = 1}^{q} \frac{(p+q-2i)!(p+q-2i+1)!}{(p-i)!(q-i)!}  \right]^{-1} = 2^{pq}{pq \choose n-1,n-3,\ldots,n-2q+1}\prod_{i=1}^q {p+q-2i \choose p-i,q-i}^{-1}.
    \end{equation*}
\end{thm} 

\begin{proof}
    Lemma~\ref{lem:clan-chains-vs-inv-chains} gives a $2^q$-to-$1$ correspondence between maximal chains in $\Clan_{p,q}$ and reduced involution words of $w_{p,q}$ which preserves the labeling of covering relations,
    \begin{equation} \label{eq:all-clan-stanley}
        \sum_{\substack{\gamma \in \Clan_{p,q} \\ \text{$\gamma$ matchless}} } F_\gamma = 2^q \hat F_{w_{p,q}} = 2^q P_{(n-1,n-3,\ldots,n-2q+1)},
    \end{equation}
    where the second equality holds by Lemma~\ref{lem:dominant}.

    Let $g^{\lambda}$ denote the number of \emph{unmarked} standard shifted tableaux of shape $\lambda$: fillings of the shifted shape of $\lambda$ by $1, 2, \ldots, |\lambda|$ which are strictly increasing across rows and down columns. As in Proposition~\ref{prop:reduced-word-count}, the coefficient of $x_1 x_2 \cdots x_{pq}$ on the lefthand side of \eqref{eq:all-clan-stanley} is the number of maximal chains in $\Clan_{p,q}$, while the coefficient on the right side is $2^q 2^{|\lambda|-\ell(\lambda)} g^{\lambda} = 2^{pq} g^{\lambda}$.  The \emph{shifted hook length formula} \cite{thrall-shifted-hook-length} computes $g^{\lambda}$, as follows. The \emph{doubled shape} $\tilde\mu$ of a strict partition $\mu$ is obtained by placing a copy of the shifted shape of $\mu$ to the right of its transpose so that their main diagonals are adjacent (but are not identified): 
    \begin{equation*}
        \mu = \young(\hfil\hfil\hfil\hfil,:\hfil\hfil\hfil,::\hfil) \qquad \leadsto \qquad \tilde\mu =  \young(\cdot\hfil\hfil\hfil\hfil,\cdot\cdot\hfil\hfil\hfil,\cdot\cdot\cdot\hfil,\cdot\cdot)
    \end{equation*}
    where $\cdot$ marks the new boxes. The shifted hook length formula is then $g^{\mu} = |\mu|! / \prod_{(i,j) \in \mu} h_{ij}$, where $h_{ij}$ is the usual hook length of box $(i,j)$ in $\tilde{\mu}$, but $(i,j)$ only runs over those boxes corresponding to the original shifted shape $\mu$. In the example above, $g^{\mu} = 8!/(7\cdot 5 \cdot 4 \cdot 2 \cdot 4 \cdot 3 \cdot 1 \cdot 1)$.

    When $\lambda = (p+q-1, p+q-3, \ldots, p-q+1)$, this formula gives
    \begin{equation*}
        2^{pq} g^{\lambda} = 2^{pq} (pq)!\left[\prod_{i = 1}^{q-1} 2^{q-i}\frac{(p+q-2i)!}{(p-i)!} \prod_{i=1}^q \frac{(p+q-2i+1)!}{2^{q-i}(q-i)!}  \right]^{-1},
    \end{equation*}
    where the $i$\th factor in the first product is the product of the hook lengths in row $i$ and columns $1, \ldots, q-1$, and the $i$\th factor in the second product is the product of the remaining hook lengths in row $i$.

\end{proof}

As a corollary of Theorem~\ref{thm:maximal-chain-count} we obtain an interesting symmetric function identity, which also appears in \cite[\sectionsymbol 4.6]{dewitt} and \cite[\sectionsymbol 7]{worley} in a slightly different form.
\begin{cor}
    \begin{equation*}
        \sum_{\lambda \subseteq [p] \times [q]} s_{\lambda} s_{(\lambda^\vee)^t} = Q_{(p+q-1, p+q-3, \ldots, p-q+1)}.
    \end{equation*}
\end{cor}

\begin{proof}
    When $\gamma$ is matchless, $F_{\gamma} = s_{\lambda^+(\gamma)} s_{\lambda^-(\gamma)}$ by Theorem~\ref{thm:matchless-clan-stanley}, where $\lambda^+(\gamma)_i$ is the number of $-$'s following the $i$\th $+$ in $\gamma$, and $\lambda^-(\gamma)_i$ is the number of $+$'s following the $i$\th $-$. Now,
    \begin{align*}
        \lambda^-(\gamma)^t_j &= \#\{i : \lambda^-(\gamma)_i \geq j\}\\
        &= \text{number of $-$'s followed by at least $j$ $+$'s}\\
        &= q-(\text{number of $-$'s followed by at most $(j{-}1)$ $+$'s})\\
        &= q-(\text{number of $-$'s following the $(p{-}j{+}1)$\th $+$})\\
        &= q - \lambda^+(\gamma)_{p-j+1} = \lambda^+(\gamma)^\vee_j.
    \end{align*}
    That is, $\lambda^-(\gamma)^t = \lambda^+(\gamma)^\vee$. The map $\gamma \mapsto \lambda^+(\gamma)$ is a bijection between matchless $(p,q)$-clans and partitions contained in $[p] \times [q]$, so the corollary follows from Theorem~\ref{thm:maximal-chain-count}.

\end{proof}

\bibliographystyle{plain}
\bibliography{../bib/algcomb}

\begin{thebibliography}{10}

\bibitem{BGG-schubert}
I.~N. Bernstein, I.~M. Gel'fand, and S.~I. Gel'fand.
\newblock Schubert cells and cohomology of the space {G/P}.
\newblock {\em Russian Math. Surveys}, 28(3):1--26, 1973.

\bibitem{billeyjockuschstanley}
S.~Billey, W.~Jockusch, and R.~P. Stanley.
\newblock Some combinatorial properties of {Schubert} polynomials.
\newblock {\em J. Algebraic Combin.}, 2:345--374, 1993.

\bibitem{brion}
M.~Brion.
\newblock The behaviour at infinity of the {Bruhat} decomposition.
\newblock {\em Comment. Math. Helv.}, 73:137--174, 1998.

\bibitem{CJW}
M.~Can, M.~Joyce, and B.~Wyser.
\newblock Chains in weak order posets associated to involutions.
\newblock {\em J. Combin. Theory Ser. A}, 137:207--225, 2016.

\bibitem{dewitt}
E.~DeWitt.
\newblock {\em Identities Relating Schur s-Functions and Q-Functions}.
\newblock PhD thesis, University of Michigan, 2012.

\bibitem{HMP2}
Z.~Hamaker, E.~Marberg, and B.~Pawlowski.
\newblock Involution words {II}: braid relations and atomic structures.
\newblock {\em J. Algebraic Combin.}, 45:701--743, 2017.

\bibitem{HMP1}
Z.~Hamaker, E.~Marberg, and B.~Pawlowski.
\newblock Involution words: counting problems and connections to {Schubert}
  calculus for symmetric orbit closures.
\newblock {\em J. Combin. Theory Ser. A}, to appear.

\bibitem{HMP4}
Z.~Hamaker, E.~Marberg, and B.~Pawlowski.
\newblock Schur {P}-positivity and involution {Stanley} symmetric functions.
\newblock {\em Int. Math. Res. Notices}, to appear.

\bibitem{hultman-twisted-involutions}
A.~Hultman.
\newblock The combinatorics of twisted involutions in {Coxeter} groups.
\newblock {\em Trans. Amer. Math. Soc.}, 359(6):2787--2798, 2007.

\bibitem{lascouxschutzenbergerschubert}
A.~Lascoux and M-P. Sch{\"u}tzenberger.
\newblock Polyn{\^o}mes de {Schubert}.
\newblock {\em Comptes Rendus des S{\'e}ances de l'Acad{\'e}mie des Sciences.
  S{\'e}rie I. Math{\'e}matique}, 294:447--450, 1982.

\bibitem{macdonald}
I.~Macdonald.
\newblock {\em Symmetric Functions and Hall Polynomials}.
\newblock Oxford University Press, 1995.

\bibitem{M2}
Ian~G. Macdonald.
\newblock {\em {Notes on {S}chubert Polynomials}}, volume~6.
\newblock Publications du LACIM, Universit\'e du Qu\'ebec \`a Montr\'eal, 1991.

\bibitem{manivel}
L.~Manivel.
\newblock {\em Fonctions sym{\'e}triques, polyn{\^o}mes de Schubert et lieux de
  d{\'e}g{\'e}n{\'e}rescence}.
\newblock Soci{\'e}t{\'e} Math{\'e}matique de France, 1998.

\bibitem{matsuki-oshima-clans}
T.~Matsuki and T.~\=Oshima.
\newblock Embeddings of discrete series into principal series.
\newblock In {\em The Orbit Method in Representation Theory}, volume~82 of {\em
  Progress in Mathematics}, 1990.

\bibitem{pittel-romik}
B.~Pittel and D.~Romik.
\newblock Limit shapes for random square {Young} tableaux.
\newblock {\em Adv. in Applied Math.}, 38:164--209, 2007.

\bibitem{richardson-springer}
R.~W. Richardson and T.~A. Springer.
\newblock The {Bruhat} order on symmetric varieties.
\newblock {\em Geom. Dedicata}, 35:389--436, 1990.

\bibitem{stanleysymm}
R.~P. Stanley.
\newblock On the number of reduced decompositions of elements of {Coxeter}
  groups.
\newblock {\em European J. Combin.}, 5:359--372, 1984.

\bibitem{thrall-shifted-hook-length}
R.~M. Thrall.
\newblock A combinatorial problem.
\newblock {\em Michigan Math. J.}, 1:81--88, 1952.

\bibitem{wachs-schubert}
M.~Wachs.
\newblock Flagged {Schur} functions, {Schubert} polynomials, and symmetrizing
  operators.
\newblock {\em J. Combin. Theory Ser. A}, 40:276--289, 1985.

\bibitem{worley}
D.~Worley.
\newblock {\em A theory of shifted Young tableaux}.
\newblock PhD thesis, Massachusetts Institute of Technology, 1984.

\bibitem{wyser-clans}
B.~J. Wyser.
\newblock Schubert calculus of {Richardson} varieties stable under spherical
  {Levi} subgroups.
\newblock {\em J. Algebraic Combin.}, 38:829--850, 2013.

\bibitem{wyser-yong-clans}
B.~J. Wyser and A.~Yong.
\newblock Polynomials for $\operatorname{GL}_p \times \operatorname{GL}_q$
  orbit closures in the flag variety.
\newblock {\em Selecta Math.}, 20:1038--1110, 2014.

\bibitem{wyser-yong}
B.~J. Wyser and A.~Yong.
\newblock Polynomials for symmetric orbit closures in the flag variety.
\newblock {\em Transform. Groups}, 22:267--290, 2017.

\end{thebibliography}

\end{document}